\documentclass{article}

\title{On global behaviour of classical effective field theories}
\author{Istvan Kadar}

\usepackage[T1]{fontenc}
\usepackage[utf8]{inputenc}
\usepackage[english]{babel}
\usepackage{lmodern}
\usepackage[margin=3cm]{geometry}
\usepackage[backend=biber,style=ieee-alphabetic,sorting=nty,isbn=false,url=false,eprint=true]{biblatex}

\usepackage{alltt}
\usepackage{amsfonts}
\usepackage{amsmath}
\usepackage{amssymb}
\usepackage{amsthm}
\usepackage{enumitem}
\usepackage{tikz-cd} 	
\usepackage[nottoc]{tocbibind}
\usepackage{graphicx}
\usepackage{caption}
\usepackage{subcaption}
\usepackage{aliascnt}
\usepackage{hyperref}
\usepackage{cases}
\usepackage{comment}	
\usepackage[capitalize,nameinlink]{cleveref}
\usepackage{tikz}
\usetikzlibrary{tikzmark}

\theoremstyle{definition}

\newtheorem*{claim}{Claim}

\newtheorem*{theorem*}{Theorem}
\newtheorem*{lemma*}{Lemma}

\theoremstyle{plain}
\newtheorem{theorem}{Theorem}[section]

\newtheorem{lemma}[theorem]{Lemma}
\newtheorem{prop}[theorem]{Proposition}

\theoremstyle{definition}
\newtheorem{definition}[theorem]{Definition}

\theoremstyle{remark}
\newtheorem{remark}[theorem]{Remark}

\crefname{lemma}{Lemma}{Lemmas}
\crefname{prop}{Proposition}{Proposition}
\crefname{conjecture}{Conjecture}{Conjecture}
\crefname{cor}{Corollary}{Corollary}
\crefname{remark}{Remark}{Remark}
\crefname{defi}{Definition}{Definition}

\newcommand{\C}{\mathbb{C}}

\newcommand{\R}{\mathbb{R}}

\newcommand{\T}{\mathbb{T}}
\newcommand{\Z}{\mathbb{Z}}

\newcommand{\abs}[1]{\left\lvert #1\right\rvert}

\newcommand{\jpns}[1]{\langle #1 \rangle}

\newcommand{\norm}[1]{\left\lVert #1\right\rVert}

\renewcommand{\d}{\mathrm{d}}

\DeclareMathOperator{\supp}{supp}

\AtEveryBibitem{\clearlist{language}}
\AtEveryBibitem{\clearfield{note}}
\AtEveryBibitem{\clearfield{urldate}}  
\AtEveryBibitem{\clearfield{month}}   

\crefformat{equation}{(#2#1#3)}
\numberwithin{equation}{subsection}
\newcommand{\sob}[2]{\norm{#1}_{H^{#2}}}

\newcommand{\intt}{\int_1^t\d s}
\newcommand{\intphi}{\int_1^t\d s \int \d\nu e^{is\phi}}
\newcommand{\intl}{\int_1^t\d s e^{-is\jpns{D}}}
\newcommand{\inth}{\int_1^t\d s e^{-is\jpns{D}_M}}
\renewcommand{\d}{\text{d}}

\begin{document}
	\maketitle
	\begin{abstract}
	We continue the rigorous study of classical effective field theories (EFTs) that was recently initiated in the work of Reall and Warnick \cite{reall_effective_2022}. We study a system with one light and one heavy field with cubic coupling and prove global existence (of the UV solution) under an effective norm in the high mass limit. Furthermore, we prove that the global solution linearly scatters (in Lax-Philips sense) and that this final state has an expansion in inverse powers of the mass, and can be recovered from the EFT equation alone.
	\end{abstract}
	\setcounter{tocdepth}{1}
	\tableofcontents
	\section{Introduction}
	\subsection{Effective field theories and singular limits}
	In this article, we will investigate the equation of motion corresponding to the Lagrangian
	\begin{equation}
		\begin{gathered}
			\mathcal{L}=\frac{1}{2}(\partial \bar U\cdot \partial \bar U +\bar U^2)+\frac{1}{2}(\partial \bar V\cdot\partial \bar V+M^2\bar{V}^2)+W(\bar{U},\bar{V}),
		\end{gathered}
	\end{equation}
	where $\bar U,\bar V:\R_t\times\R_x^{3}\to\R$ are the light and heavy fields, $W:\R^2\to\R$ is the nonlinear potential that couples $\bar{U}$ to $\bar{V}$, and  $\partial\bar{U}\cdot\partial\bar{V}=-\partial_t\bar{U}\partial_t\bar{V}+\partial_{x_i}\bar{U}\partial_{x_i}\bar{V}$. For concreteness, we will take $W(\bar{U},\bar{V})=\frac{\bar{U}^2\bar{V}}{2}$, but the methods are applicable to more general potentials that satisfy
	\begin{itemize}
		\item $W$ is smooth in its argument.
		\item $W^{-1}(0)=(0,0)$.
		\item  at least cubic, ie. $\abs{W}\leq (U^2+V^2)^{3/2}$.
	\end{itemize}
	\begin{remark}
		The first condition may be relaxed to some $C^k$ regularity, depending on the number of derivatives our argument uses. The second requirement has the physical interpretation of a trivial vacuum. \footnote{Note that this may be dropped by some assumption on the decay of initial data as in \cite{tao_nonlinear_2006} Exercise 6.18} The last point is important for the non-linearity to be perturbative, ie. to have sufficient decay.\footnote{For non-trivial effective field theory (to be defined below) one needs $W$ such that $\partial_{\bar{V}}W|_{\bar{V}=0}(U)$ is not zero around the origin, because the PDE methods only work at tree level \cite{skinner_david_advanced_2020}.}
	\end{remark}
	
	At the level of the equation of motion, we are interested in the $M\to\infty$ limit for
	\begin{equation}\label{original equation of motion}
		\begin{gathered}
			(\Box-1)\bar U=\bar U\bar V\\
			(\Box-M^2)\bar V=\frac{\bar U^2}{2}\\
			\bar U(0)=\bar U_0, \quad\partial_t\bar U(0)=\bar U_1,\quad \bar V(0)=\bar V_0,\quad \partial_t\bar V(0)=\bar V_1
		\end{gathered}
	\end{equation}
	where $\Box=-\partial_t^2+\partial_{x_i}\partial_{x_i}$.
	
	This equation is meant to represent the interaction of a light ($\bar{U}$) and a heavy ($\bar{V}$) field and is called an ultraviolet (UV) equation. The heuristic is that, in the $M\to\infty$ limit, under some conditions, one can "integrate out" the heavy field and be left with an equation for the light one (\cite{burgess_introduction_2020}). On a purely formal basis, one performs a binomial expansion on $\bar{V}=(\Box-M^2)^{-1}\bar{U}^2/2$ treating "$\Box\ll M^2$" and truncate after some finite number of terms. Then, one plugs $\bar{V}$ in the equation for $\bar{U}$ to get a higher order PDE, which we call an effective field theory (EFT) for $\bar{U}$. The aim of this paper is to shed some more light on the rigorous foundations of this idea. Our main results are global existence and scattering for the above equation under certain restrictions on the initial data that are natural in this context (see \cite{reall_effective_2022}).
	
	\begin{theorem*}[Main theorem, schematic. \cref{global existence}, \cref{scattering lemma}]
		The system \ref{original equation of motion}, with initial data $\bar{U}_0,\bar{V}_0,\bar{U}_1,\bar{V}_1$ with finite energy in an \textit{effective} norm admits a global  solution for all $M>M_0$ where $M_0$ depends on the initial energy. In particular, $\bar{U}$ can have $\mathcal{O}(1)$ $H^s$ norm. Moreover, the light field $\bar{U}$ scatters to the solution of an effective field theory (EFT) with the same initial data.
	\end{theorem*}

	Effective field theories are important to the current understanding of our physical models. For a recent comprehensive guide see \cite{burgess_introduction_2020}. Non-renormalisable theories at the dawn of quantum field theory were treated as \textit{bad} theories, since they did not make sense in the high frequency limit. A typical example for such a theory is the Fermi 4 point interaction. In physics parlance, the interaction term between the 4 fermions in the Lagrangian is not renormalisable, as its coupling has negative mass dimension, therefore it is not easy to interpret what predictions the theory has for high energy particles.\footnote{The resolution of the apparent problems was the introduction of electroweak theory with the W and Z bosons, see chapter 7 \cite{burgess_introduction_2020} for a comprehensible guide. } Nevertheless the theory had great predictive success in the low energy (perturbative) setting, and it is important to understand when and how these \textit{bad} theories are able to give valid predictions. Moreover, our current understanding of gravity as a quantum field theory (treated with the naive canonical quantisation scheme) is non renormalisable, thus the 
	need for its treatment as an EFT \cite{burgess_quantum_2004}.
	
	These physical motivations (and more) lay the ground for a recent rigorous study of EFTs \cite{reall_effective_2022}, where the authors consider the equation of motion derived from the Lagrangian ${\partial_\mu\phi\partial^\mu\phi+\frac{M^2}{8}(\abs{\phi}^2-1)^2}$ in the limit $M\to\infty$, for $\phi:\T^n\to\C$. The regime of applicability of both the classical approximation and the EFT framework at first glance seem not to overlap, which is why they introduce conditions on initial data. We refer the reader to their paper for detailed explanation. This paper is meant as a natural continuation of their work on the mathematical foundations of the subject. In current mathematical PDE terminology, they were able to
	\begin{itemize}
		\item give a definition of an EFT to capture expected physical properties
		\item prove ($M$ independent) local existence for the UV equation with initial data that heuristically correspond to the heavy degree of freedom being in its ground state
		\item prove uniqueness of the EFT solutions (which is a solution modulo controlled error terms)
		\item compare EFT and UV solutions.
	\end{itemize}
	As mentioned above, we extend these for a different, though similar system with 2 extra points
	\begin{itemize}
		\item global existence of solutions under EFT conditions
		\item scattering theory of EFTs.
	\end{itemize}
	
	On the more mathematical side, the study of EFTs resembles the study of singular limits (see eg. \cite{masmoudi_nonlinear_2002},\cite{pasquali_dynamics_2019},\cite{machihara_nonrelativistic_2002},\cite{klainerman_singular_1981} ), but where the limit is not taken, rather an expansion in the smallness parameter is sought. As an example of singular limit, one can analyse the non-relativistic limit of the non-linear Klein-Gordon equation
	\begin{gather*}\label{Klein}
		\Big(\frac{1}{c^2}\partial_{tt}-\Delta+\frac{m^2c^2}{\hbar^2}\Big)u_c=f(u_c)
	\end{gather*}
	$c\to\infty$. As shown in \cite{masmoudi_nonlinear_2002}, one recovers a system of Schrodinger equation for $v=(v_+,v_-)$
	\begin{gather*}\label{Sch}
		\Big(i\partial_t +\frac{\hbar}{2m}\Delta\Big)v=\tilde{f}(v)
	\end{gather*}
	where $u_c=e^{ic^2t}v_++e^{-ic^2t}v_-+o(1)$, and $\tilde{f}$ depends on $f$.
	In such singular limits, when one is effectively ignoring high frequency part of the solution, the non-linearities can either pass to a limit (eg. for the example above, cubic term $f(u)=\abs{u}^2u$ gives cubic limit $\tilde{f}_{\pm}(v)=(\abs{v_\pm}^2+2\abs{v\mp}^2)v_\pm$), or if one takes a weak limit, new forcing terms might be introduced as in the case of Burnett's conjecture (\cite{huneau_high-frequency_2018},\cite{burnett_high-frequency_1989}) or an inverted pendulum (\cite{evans_weak_2016}).
	
	Our aim is not to pass to a limit (which would be linear Klein-Gordon $(\Box-1)\bar{U}=0$), instead given equations of motion such as \cref{original equation of motion}	we want to have a quantitative control on the error term in the $M\to\infty$ limit of an expansion as
	\begin{equation*}
		\begin{gathered}
			(\Box-1)\bar{U}=\sum_{i=1}^n\frac{F_i[\bar{U}]}{M^i}+\frac{\text{error}}{M^{n+1}}.
		\end{gathered}
	\end{equation*}
	where $F_i$ are local functionals of $\bar{U}$.	In a sense, this difference is one between quantitative and qualitative description of the limit, and so our methods of study differ considerably from the singular limits context, therefore we will pursue no further connections.
	
	\subsection{Previous works}\label{previous works}
	Equation \ref{original equation of motion} is a system of non-linear Klein-Gordon equations. Concerning global existence, there are two main lines of study. Either one considers small initial data in a fixed (maybe high regularity) norm, where dispersion gives decay and the nonlinearity will (hopefully) be perturbative, or aims to prove a continuation result at low regularity and use a conserved quantity to get global solutions. We focus on the first of these, but see \cite{tao_nonlinear_2006} for some results in the second direction.
	
	One of the first works for small data problems was Shatah's work (\cite{shatah_global_1982}) to treat the cases where there is plenty of decay to control the solutions without a detailed analysis. In particular, he did not consider the case of quadratic forcing in $\R^{3+1}$. This was later completed by him \cite{shatah_normal_1985} introducing the method of normal forms and independently by Klainerman \cite{klainerman_global_1985} using hyperboloidal foliations. This was subsequently extended to $\R^{2+1}$ \cite{ozawa_global_1996} and many other settings, see the introduction of \cite{germain_global_2011} for a comprehensive review of literature. (For results in the case of non algebraic non-linearities see the introduction of \cite{keel_small_1999}.) Both of the approaches proved to be very important for treating similar problems, we mention just few results using each method (the lists are far from exhaustive). The methods of normal forms, with some extension is now known as space-time resonance method and was used to prove global existence results for non-linear Schroedinger \cite{germain_global_2008} \cite{germain_global_2012}, for gravity water waves \cite{germain_global_2012-1}, for Klein-Gordon systems \cite{germain_global_2011} \cite{ionescu_global_2014} \cite{deng_multispeed_2018}, wave Klein-Gordon systems \cite{ionescu_global_2019}. The hyperboloidal foliation was perfected by Hörmander, \cite{hormander_lectures_1997}, and reached its format of main use by LeFloch and Ma in \cite{lefloch_hyperboloidal_2015}, so to apply to numerous different dimensions and non-linearities \cite{ma_global_2017} \cite{ma_global_2019} \cite{dong_two_2022}. Both of these methods were able to prove the stability of Minkowski space with massive scalar field \cite{lefloch_global_2017},\cite{ionescu_einstein-klein-gordon_2022}.
	
	There are several other techniques to treat small data problems for hyperbolic geometric equations, but these have not been used for Klein-Gordon fields until now, so we just mention them briefly. The first is the $r^p$ method of Dafermos and Rodnianski \cite{dafermos_new_2010}. This method was mainly developed for its applicability in non-trivial geometric backgrounds such as black holes. Its usage there is hard to overestimate, see \cite{Angelopoulos2018} and references therein. For a use of the technique in Minkowski see \cite{klainerman_global_2020}. Another technique commonly used in non-trivial geometric settings is based on Melrose's b-analysis, for Minkowski application see \cite{hintz_stability_2020}.
	
	In all of the above works concerning Klein-Gordon equations, the authors prove the existence of a small number $\epsilon$, dependent on the speeds and masses in the system, such that if some norm of the initial data is less than $\epsilon$ then global solutions exist. This is in stark contrast to our goal (see below), even though the methods will be similar. The existence of a limiting parameter would suggest the use of semiclassical techniques, but these seem not to be immediately applicable to the system under consideration as the limit we take forces the two parts of the system to decay on different timescales, so one cannot rescale spacetime coordinates/introduce weights on derivatives.
	\subsection{Difficulties and main theorem}\label{problem description}
	
	The goal of the paper is to show that the restriction on initial data, that forces the heavy field to its ground state, can give global solutions for certain Lagrangians. The result here is completely perturbative - not relying on any coercive conserved quantity (energy) - thus giving a good picture of the evolution of solution and yielding a scattering statement on constant time slices. In particular, we obtain a convergent approximation of the final solution. Here, we look at a different Lagrangian than the one considered in \cite{reall_effective_2022} due to a difficulty regarding \textit{loss of derivative}.
	
	Motivated by \cite{reall_effective_2022}, we want to consider global existence for this problem with $\mathcal{O}(1)$ data thought of as an EFT, ie. for some fixed $E$ (independent of $M$)
	\begin{equation}\label{original initial data}
		\sob{\bar U_0}{N}+\sob{\bar U_1}{N-1}+M^{1+p}\sob{\bar V_0}{N-1}+M^p\sob{\bar V_0}{N}+M^p\sob{\bar V_1}{N-1}< E.
	\end{equation}
	for $p=1$. The restriction on $\bar{U}$ is the natural energy at $N$ regularity, while the one for $V$ is suppressed by a factor of $M^p$, indicating that $\bar{V}$ is in its ground state. All $p>0$ would force the heavy field to be small, but we argue below that $p=1$ is a natural/physical choice. Solving \cref{original equation of motion} with \cref{original initial data} is not a perturbative problem as it is not a small data problem for $\bar{U}$, nor is it amenable to energy conservation methods due to the lack of a positive definite potential. 
	\begin{remark}
		The problem considered in \cite{reall_effective_2022} has a good potential energy functional. Therefore, due to subcritical behaviour ($\dot{H}^{3/4}\subset L^4$) one can try to show global existence using Strichartz estimates in a rough norm. This however does not shed light on the long time behaviour of the solution in higher regularity norms, which is the aim of this study. In particular, one cannot rule out weak turbulence behaviour with standard tools.
	\end{remark}
	To introduce a small parameter, we follow the change of coordinates motivated by EFT analysis in \cite{reall_effective_2022}: setting ${V}=\bar V+\frac{\bar U^2}{2M^2}$, we get
	\begin{equation}\label{v modified}
		\begin{gathered}
			(\Box-1)\bar U=\bar U{V}-\frac{\bar U^3}{2M^2}\\
			(\Box-M^2){V}=\frac{\bar U^2{V}-\frac{\bar U^4}{2M^2}+\partial \bar U\partial \bar U}{M^2}\\
			\bar U(0)=\bar U_0, \quad\partial_t\bar U(0)=\bar U_1,\quad {V}(0)={V}_0,\quad \partial_t{V}(0)={V}_1.
		\end{gathered}
	\end{equation}
	For this equation, we expect that $V$ is a better approximation to the true heavy degree freedom than $\bar{V}$. The nonlinear force exerted by the light field on $V$ is suppressed by factors of $M$, so, when in its ground state, $V$ will be smaller (in powers of $M$) than $\bar V$ implying that $p=1$ is optimal in \cref{original initial data}, to be consistent with the light field assumption. We consider this problem with initial data 
	\begin{gather}\label{v modified initial data}
		\sob{\bar U_0}{N}+\sob{\bar U_1}{N-1}+M^3\sob{{V}_0}{N-1}+M^2\sob{{V}_0}{N}+M^2\sob{{V}_1}{N-1}< E,
	\end{gather}
	which is \underline{not} equivalent to \eqref{original initial data}, but is still an EFT data as considered in \cite{reall_effective_2022}.	Observe, that in \eqref{v modified} at time $t=0$, with initial data \eqref{v modified initial data}, the nonlinear terms are initially suppressed by many factors of $M$. In particular, for $U=\frac{\bar U}{M}$ we have a small data problem in the high $M$:
	\begin{equation}\label{eom}
		\begin{gathered}
			(\Box-1) U=U{V}- U^3/2\\
			(\Box-M^2)V= U^2V-U^4/2+\partial U\partial U\\
			U(0)=U_0,\quad \partial_t U(0)= U_1,\quad V(0)=V_0,\quad\partial_tV(1)=V_1,\\
		\end{gathered}
	\end{equation}
	with initial data satisfying
	\begin{equation}\label{initial data norm}
		M\sob{U_0}{N}+M\sob{U_1}{N-1}+M^3\sob{V_0}{N-1}+M^2\sob{V_0}{N}+M^2\sob{V_1}{N-1}<E.
	\end{equation}
	The equation to solve became a small data problem for fixed $E$ in the large $M$ limit, however standard results, such as \cite{klainerman_global_1985} \cite{ionescu_global_2014}, \cite{deng_multispeed_2018} do not straightforwardly apply, since the smallness of the data ($M$) also changes the equations under consideration. Our main result is
	
	\begin{theorem*}[\cref{global existence}]
		For $E$ fixed and $M$ sufficiently large depending only on $E$, \cref{eom} has a unique global solution if the initial data satisfies \cref{initial data norm}.
	\end{theorem*}

	Assuming linear decay, the right hand side of \eqref{eom} is bounded by ($t^{-1.5}$) in energy norm ($L^2$), although global existence is not yet immediate, since the massive field decays on timescales $t\sim M$ (\cref{decay estimate}). From the techniques presented in \ref{previous works}, we will use the space-time resonance method, as one can be very economical with the number of time derivatives taken \footnote{This is in contrast with the vector field method (hyperboloidal foliation for Klein-Gordon fields), for which the boosts $\Gamma:=x_i\partial_t-t\partial_{x_i}$ are required and would result in many factors of $M$ appearing (one needs in particular $\norm{\Gamma^5\cdot}_{L^2}$ control on initial data \cite{lefloch_hyperboloidal_2015}). } , and can obtain an explicit perturbative expansion of the solution which helps to prove scattering results.
	\begin{remark}
		The long term decay estimates proved in this paper (see \ref{expected long term behaviour}) are not sharp. The reason is that these non-sharp estimates are easier to handle using global norms (such as in \cite{germain_global_2011}), but it is likely that one could obtain sharp asymptotics using localisation (as in \cite{ionescu_global_2014}).
	\end{remark}
	
	\subsection{Linear scattering}
	Once we know that the solutions are global, we may find out what the scattering map looks like in the Lax-Phillips sense (see \cite{reed_iii_1979}). That is, given initial data $U_0,U_1,V_0,V_1\in H^n$ for \cref{eom} we can look for $\tilde{u},\tilde{v}$ solutions of the free system with initial data (called scattered state for the original problem)
	\begin{equation*}
		\begin{gathered}
			(\Box-1)\tilde{u}=0\\
			(\Box-M^2)\tilde{v}=0\\
			\tilde{u}(0)=\tilde{u}_0,\quad\partial_t\tilde{u}(0)=\tilde{u}_1,\quad\tilde{v}(0)=\tilde{v}_0,\quad\partial_t\tilde{v}(0)=\tilde{v}_1
		\end{gathered}
	\end{equation*}
	such that $\norm{U-\tilde{u}}_{H^m},\norm{V-\tilde{v}}_{H^m}\to0$ as $t\to\infty$ for some $m$. The advantage of space-time resonance method compared to others (such as the hyperboloidal foliation) is that the profiles exactly capture the above behaviour, indeed if the profiles converges their limit is given in terms of  $\tilde{u}_0,\tilde{u}_1,\tilde{v}_0,\tilde{v}_1$.
	\begin{theorem}[\cref{uniqueness of scattering state} and \cref{UV is EFT} for exact statements]\label{scattering lemma}
		Fix a solution $U,V$ of \cref{eom} with initial data satisfying EFT conditions to order $2n$ (\cref{eft condition}), and a scattering EFT\footnote{meaning that the error to satisfy the equation is integrable in time.} solution (\cref{eft sol definition}) with same data to  
		\begin{equation*}
			\begin{aligned}
				(\Box-1)u=\sum_{i=1}^{2n}\frac{F_i[u]}{M^{2i}}
			\end{aligned}
		\end{equation*}
		with scattered states $f$ and $\tilde{f}$ respectively. Here, $F_i$ are derived from \cref{eom} (see \cref{eft eom}). For $M$ sufficiently large $\norm{f-\tilde{f}}_{H^{k}}\lesssim\frac{1}{M^{2n+2}}$ for some regularity $k$.
	\end{theorem}
	
	\begin{remark}
		While this result gives some information about the long time behaviour and indeed it is the one that is called simply scattering in the literature often, it doesn't tell us everything. If we modify the linear part of the light field to a wave equation ($\Box-1\to\Box$), we know that a natural question to ask is the behaviour of the radiation field at null infinity ($\mathcal{I}$)\footnote{See \cite{dafermos_scattering_2018} for this alternative view on scattering theory}. However, information about the behaviour along $\mathcal{I}$ is hard to recover from this foliation as it corresponds to a single slice $t=\infty$.
	\end{remark}

%
	The rest of the paper is structured as follows. In \cref{preliminaries} we set up the necessary notation for the method used and give a lightning overview on how to implement it. Furthermore, we state and prove the main theorem conditionally on technical lemmas. In \cref{frequency analysis}, we provide detailed description of the resonances and prove key estimates that allow the space-time resonance method to work. \cref{a priori bound section} contains all the bounds for the nonlinear terms, in particular a strengthening of the bootstrap assumptions. There are no new conceptual estimates in this section, but all non-linearities must be bounded to improve on the constants. Finally in \cref{scattering section} we give some new global EFT definitions and prove that UV solutions (\cref{eom}) are indeed EFTs if the right initial data is chosen. Furthermore, we give a certain uniqueness result in this class.
	
	\textbf{Acknowledgement}:The author would like to thank Claude Warnick for proposing the topic of this paper and for the many helpful discussions on both the technical and conceptual parts of the work. This project was founded by EPSRC.
	
	\section{Preliminaries}\label{preliminaries}
	\subsection{Space time resonance method}\label{space time method}
	To set up the method, let's define the half waves and the profiles for the system in \eqref{eom}
	\begin{equation}\label{resonance transformation}
		\begin{gathered}[c]
			u_\pm=(\partial_t\pm i\jpns{D})U\\
			l_\pm=e^{\mp i\jpns{D}t}u_\pm\\
		\end{gathered}
		\qquad\qquad
		\begin{gathered}[c]
			v_\pm=(\partial_t\pm i\jpns{D}_M)V\\
			h_\pm=e^{\mp i\jpns{D}_Mt}v_\pm\\
		\end{gathered}
	\end{equation}
	where $\jpns{a}_M^2=\abs{a}^2+M^2, \jpns{a}=\jpns{a}_1$, $D_{x_j}=i\partial_{x_j}$ \footnote{for the theory of pseudo differential operators see \cite{hormander_lectures_1997}, and their usage closest to the present work in eg.\cite{ionescu_global_2014}}. This definition makes sense for functions in the Schwartz class ($\mathcal{S}$) and may be extended by density, so we have $U,\dot{U}\to (u_-,u_+)$ is a continuous map $H^N\times H^{N-1}\to H^{N-1}\times H^{N-1}$ (similarly for $V$). Therefore, from now on, we will assume that all functions are of Schwartz class when proving estimates. The initial data condition \cref{initial data norm} maps to
	\begin{equation}
		M\sob{v_{0,\pm}}{N-1}+\sob{u_{0,\pm}}{N-1}<E/M.
	\end{equation}
	
	Note, that $l_\pm,h_\pm$ will evolve purely according to the nonlinearity:
	\begin{gather}\label{evolution of profiles}
		\partial_t l_\pm=-e^{\mp i\jpns{D}t}(UV-U^3/2)\\
		\partial_t h_\pm=-e^{\mp i\jpns{D}_Mt}(U^2V-U^4/2+\partial U\partial U),
	\end{gather}
	that is, they represent the Duhamel/nonlinear part of the solution.
	
	Using \cref{resonance transformation}, we may express the nonlinearities in terms of the new variables as
	\begin{gather*}
		UV=\frac{e^{it\jpns{D}}l_+-e^{-it\jpns{D}}l_-}{2i\jpns{D}}\frac{e^{it\jpns{D}_M}h_+-e^{-it\jpns{D}_M}h_-}{2i\jpns{D}_M}\\
		\partial U\partial V=-\big(\frac{u_++u_-}{2}\big)^2+\nabla_j\frac{u_+-u_-}{2i\jpns{D}}\nabla_j\frac{u_+-u_-}{2i\jpns{D}}.
	\end{gather*}
	Therefore the quadratic contribution to the evolution of the profiles can be expressed in terms of the profiles
	\begin{equation}\label{partial of l}
		\begin{gathered}
			\partial_t \hat l_{+(\text{quadratic})}(\rho):=-e^{- i\jpns{D}t}UV=-\sum_{\epsilon_1,\epsilon_2=\pm}\int\d\nu e^{it\phi^{l,l,h}_{+,\epsilon_1,\epsilon_2}}a^{l,l,h}_{+,\epsilon_1,\epsilon_2}l_{\epsilon_1}(\nu)h_{\epsilon_2}(\rho-\nu)
		\end{gathered}
	\end{equation}
	where $\epsilon_1,\epsilon_2\in\{\pm1\}$, $\phi^{l,l,h}_{+,\epsilon_1,\epsilon_2}=-\jpns{\rho}+\epsilon_1\jpns{\nu}+\epsilon_2\jpns{\rho-\nu}_M$ is a phase and $a^{l,l,h}_{+,\epsilon_1,\epsilon_2}=\frac{1}{4}\frac{\epsilon_1\epsilon_2}{\jpns{\rho-\nu}_M\jpns{\nu}}$ acts as a multiplier. And similarly for the heavy field
	\begin{gather}\label{partial of h}
		\partial_t \hat h_{+(\text{quadratic})}(\rho)=e^{-i\jpns{D}t}(\widehat{\partial_t u\partial_t u}-\widehat{\partial_i u\partial_i u})=\sum_{\epsilon_1,\epsilon_2=\pm}\int\d\nu e^{it\phi^{h,l,l}_{+,\epsilon_1,\epsilon_2}}a^{h,l,l}_{+,\epsilon_1,\epsilon_2}l_{\epsilon_1}(\nu)h_{\epsilon_2}(\rho-\nu)
	\end{gather}
	where $\phi^{h,l,l}_{+,\epsilon_1,\epsilon_2}=-\jpns{\rho}_M+\epsilon_1\jpns{\nu}+\epsilon_2\jpns{\rho-\nu}$ and $a^{h,l,l}_{+,\epsilon_1,\epsilon_2}=\frac{1}{4}\Big(1-\epsilon_1\epsilon_2\frac{\nu\cdot(\rho-\nu)}{\jpns{\nu}\jpns{\rho-\nu}}\Big)$.
	From now on, to ease notation, we will not write out $\epsilon$ indices, multipliers ($a$) or summations. Therefore in subsequent calculation the form of $\phi$ and $a$ will be implicitly assumed as will the propagation of $l$ and $h$. This should not be confusing, as such terms are trivial to recover once all the appropriate $\epsilon$ are inserted. Similarly, since $a$ is symbolic, ie. it has good decay properties $\abs{\partial^\alpha a}<(\abs{\rho}+\abs{\nu})^{-\abs{\alpha}}$, it plays no obstruction against using multiplier theory (Coifman-Meyer theory), and we do not want to use any gain of derivatives or other structure that $a$ may provide, thus we will drop them for sake of simplicity. The frequencies $\phi$ will nevertheless play an important role and for sake of concreteness, we choose $\phi_u=-\jpns{\rho}+\jpns{\nu}_M-\jpns{\rho-\nu}$ and $\phi_v=-\jpns{\rho}_M+\jpns{\rho-\nu}+\jpns{\nu}$, since as explained in \cref{frequency analysis} all other choices of sign can be treated by the same techniques. Similarly, we'll focus on the quadratic part of the equation only, as the higher order terms have more decay thus easier to treat. In particular, one can improve the bootstrap estimates (see below) by only using \cref{integral estimate} and sidestepping the frequency analysis.
	
	We will consider the quantities
	\begin{equation}
		\begin{gathered}
			\hat{L}=\int_{1}^{t}\d s \int \d \nu e^{is\phi}\hat{l}(\nu)\hat{h}(\rho-\nu)\\
			\hat{H}=\int_{1}^{t}\d s \int \d \nu e^{is\phi}\hat{l}(\nu)\hat{l}(\rho-\nu)
		\end{gathered}
	\end{equation}
	which are the integrated nonlinear change of the profiles ($l,h$) after time $1$ capturing the perturbative effect of the nonlinearity. The reason for bounding the integral that start from 1 is only to avoid $s<1$ case, which are treated separately in all cases using a local existence result \cref{local existence}.

	The method of space time resonance builds on splitting $1(\rho,\nu)=\chi_S(\rho,\nu)+\chi_T(\rho,\nu)$, $\chi_{S,T}\in C^\infty(\R^{2n}\to[0,1])$, such that $\abs{\phi}|_{\supp\chi_S}>0$ and $\abs{\nabla_\nu\phi}|_{\supp\chi_T}>0$. Then, we may split the above integral expression
	\begin{equation}
		\begin{gathered}
			\intphi \hat{l}(\nu)\hat{h}(\rho-\nu)=\intphi (\chi_S+\chi_T)\hat{l}(\nu)\hat{h}(\rho-\nu)=:I+II.
		\end{gathered}
	\end{equation}
	\begin{remark}[Nomenclature]
		We call the part $\phi=0$ the time resonant set, and $\chi_S$ localises away from this set, in particular $\chi_S|_{\{\nabla_\nu\phi=0\}}=1$ so it has maximum value around the space resonant set ($\nabla_\nu\phi=0$). Therefore, we will refer to I (II) as space resonant (time resonant) part.
	\end{remark}
	
	Away from times resonances (ie. potentially near space resonant part), we may integrate by parts with respect to $s$ using $\frac{1}{i\phi}\partial_s e^{is\phi}=e^{is\phi}$:
	\begin{equation}
		\begin{gathered}
			I=\int_1^t\d s T_{\chi_S e^{is\phi}}(l,h)\\=e^{-it\jpns{D}}T_{m_1}(u,v)|_t-e^{-i\jpns{D}}T_{m_1}(u,v)|_1-\int_1^t\d s \Big(T_{m_1e^{is\phi}}(l,\partial_sh)+T_{m_1e^{is\phi}}(\partial_sl,h)\Big)
		\end{gathered}
	\end{equation}
	where $m_1=\frac{\chi_S}{i\phi}$ and 
	\begin{equation*}
		\begin{aligned}
			T_{m(\rho,\nu)}(f,g)=\mathcal{F}^{-1}\int\d\nu m(\rho,\nu)\hat{f}(\nu)\hat{g}(\rho-\nu)\\
			\mathcal{F}f(\rho)=\hat{f}(\rho)=\int\d x e^{ix\cdot\rho}f(x).
		\end{aligned}
	\end{equation*}
	With this, we split up $I$ into boundary terms, that are not integrated in time, and a bulk part which is cubic thanks to purely nonlinear evolution of profiles \cref{evolution of profiles}. 
	
	Away from space resonance, we may integrate by parts with respect to $\nu$ using $\frac{\nabla_\nu\phi}{s\abs{\nabla_\nu\phi}^2}\nabla_\nu e^{is\phi}=ie^{is\phi}$. When $\nabla_\nu$ acts on a profile, we may use $\nabla_\nu\hat{f}=\widehat{ixf}$ to write the result compactly as
	
	\begin{equation}
		\begin{gathered}
			II=-\int_1^t\d s \frac{1}{s}\Big(T_{e^{is\phi}\nabla_\nu\cdot m_2}(l,h)+T_{e^{is\phi}m_2}(ixl,h)+T_{e^{is\phi}m_2}(l,ixh)\Big),
		\end{gathered}
	\end{equation}
	where $m_2=\frac{\chi_T\nabla_\nu\phi}{\abs{\nabla_\nu\phi}^2}$ is a vector and is contracted with $x$ where it appears. Since our estimate will not use the sign of $x$ and are not meant to be sharp, we will drop the $i$ factor, and treat as implicit. We will prove symbolic estimates for $m_2$ which will pass down to $\nabla_\nu m_2$ and furthermore, $l,h$ have better decay behaviour than $xl,xh$ we will not worry about the first term, and will only bound the rest.
	
	To state the main estimates used, let's introduce the notation
	\begin{equation*}
		\begin{aligned}
			A\lesssim_{a,b,c...} B
		\end{aligned}
	\end{equation*}
	to mean the existence of a universal constant $C_{a,b,c...}$ depending only on $a,b,c...$, such that $A\leq C_{a,b,c...}B$. The tools to obtain bounds on $L$ and $H$ are the decay estimates (\cref{decay estimate}),
	\begin{equation}
		\norm{e^{i\jpns{D}_Mt}f}_{L^p}\lesssim_{p}\jpns{\frac{t}{M}}^{3(1/p-1/2)}\norm{f}_{W^{4(1/2-1/p),p^\prime}}
	\end{equation}
	and the product estimates (\cref{multiplier lemma}):
	\begin{equation}
		\begin{gathered}
			\norm{T_{\frac{\chi_S}{\phi}}(f,g)}_{W^{k,r}}\lesssim_{k,r,p,q,\bar{p},\bar{q}}\norm{f}_{W^{\mathfrak{a},p}}\norm{g}_{W^{k,q}}+\norm{f}_{W^{k,\bar{p}}}\norm{g}_{W^{\mathfrak{a},\bar{q}}}\\
			\norm{T_{\frac{\chi_T\nabla_\nu\phi}{\abs{\nabla_\nu\phi}^2}{\phi}}(f,g)}_{W^{k,r}}\lesssim_{k,r,p,q,\bar{p},\bar{q}}\norm{f}_{W^{\mathfrak{a},p}}\norm{g}_{W^{k,q}}+\norm{f}_{W^{k,\bar{p}}}\norm{g}_{W^{\mathfrak{a},\bar{q}}}
		\end{gathered}
	\end{equation}
	for any $f,g\in\mathcal{S}$, $k\geq0$, $1/p+1/p^\prime=1$ and $\frac{1}{r}=\frac{1}{p}+\frac{1}{q}=\frac{1}{\bar{p}}+\frac{1}{\bar{q}}$ with $\mathfrak{a}$ some fixed number. As usual, these bounds extend to the appropriate function spaces using density.
	
	\subsection{Expected long term behaviour, norms, bootstrap assumption}\label{expected long term behaviour}
	We will prove boundedness of high regularity energies ($L^2$ based) and decay estimates in rougher norms. To motivate the norms used, let's consider the effect of the non-linearity on each of these.
	
	Concerning decay, for the linear system we have 
	\begin{equation}
		\begin{gathered}
			\norm{u}_{W^{k,p}}\lesssim t^{3(\frac{1}{2}-\frac{1}{p})}\text{data}
		\end{gathered}
	\end{equation}
	for $p\geq2$. This will be too much to prove for all possible values of $p$ in the nonlinear setting. To see why, consider a time resonant integral with $p=\infty$, where we are aiming for a $t^{-1.5}$ bound on
	\begin{equation}
		\begin{gathered}
			I=\norm{\int\d s e^{is\jpns{D}}\frac{1}{s}T_{\frac{\chi_T\nabla_\nu\phi}{\abs{\nabla_\nu\phi}}}(xl,h).
			}_{W^{k,\infty}},
		\end{gathered}
	\end{equation}
	but using Minkowski inequality, and the decay \cref{decay estimate}, we only get
	\begin{equation}
		\begin{gathered}
			I\lesssim\int\frac{1}{s(t-s)^{1.5}}\norm{T_{\frac{\chi_T\nabla_\nu\phi}{\abs{\nabla_\nu\phi}}}(xl,h)}_{W^{k^\prime,1}},
		\end{gathered}
	\end{equation}
	for some $k^\prime>k$. Since the $W^{k^\prime,1}$ norm of a quadratic expression is not expected to decay, using the integrals from \cref{integral estimate}, the best expected estimate is $t^{-1}$. Tracing the numerology in above gives necessary condition $p\leq12$.
	
	For boundedness of energy ($H^N$ norms), we cannot use integration by parts arguments, because we would lose derivatives. Considering a typical forcing such as $u^2$ (dropping multipliers momentarily), we simply use Minkowski and Holder inequality
	\begin{equation}
		\norm{\intt uu}_{H^N}\lesssim\int_1^t\norm{u}_{H^N}\norm{u}_{L^\infty}.
	\end{equation}
	This integral will converge if $\norm{u}_{L^\infty}$ is integrable, so via Sobolev embedding, we need that $\norm{u}_{W^{k,p}}$ is integrable for at least some $p$, yielding $p>6$. In summary, we define the following global norms on $\R_t\times\R^3_x$
	\begin{equation}
		\begin{gathered}
			\norm{u,v}_Z=\norm{\jpns{t}^{1+3\delta}u}_{L^\infty_tW_x^{k,(\frac{1}{6}-\delta)^{-1}}}+\norm{\jpns{t/M}^{1+3\delta}\sqrt{M}u}_{L^\infty_tW_x^{k,(\frac{1}{6}-\delta)^{-1}}}\\
			\norm{l,h}_S=\norm{xl}_{L^\infty_tH^{k+3/2}_x}+\sqrt{M}\norm{xh}_{L^\infty_tH^{k+3/2}_x}\\
			\norm{u,v}_N=\norm{u}_{L^\infty_tH^{N}_x}+\sqrt{M}\norm{v}_{L^\infty_tH^{N}_x}\\
			\norm{u,v}_X=\norm{u,v}_Z+\norm{l(u),h(v)}_S+\norm{u,v}_N.
		\end{gathered}
	\end{equation}
	Furthermore, let $\norm{\cdot}_{X(I)}$ refer to the above norm restricted to $I\times\R^3$ for an interval $I\subset\R$. These are very similar to those found in \cite{germain_global_2011}, in particular it contains a mixture of bounds on the solution and the profile. We also define a bootstrap assumption
	\begin{equation}\label{bootstrap assumption}
		\norm{u,v}_{X_{I}}\leq 10 \frac{E}{M}	
	\end{equation}
	for 
	\begin{equation}\label{N k s condition}
		\begin{gathered}
			N>k+6,k>s+2,s>\mathfrak{a},
		\end{gathered}
	\end{equation}
	$\delta=1/14$ (eg. $N=70,k=50$). 
	
	\begin{remark}[Global vs dyadic norms]
		In this work, we choose to use global norms as above compared to dyadic norms such as in \cite{ionescu_global_2014} because we hope this makes the rest of the paper easier to follow. This however leads to an extra difficulty that we have to overcome, because we cannot use the fine structure of the inputs in all of the estimates, instead we need to have good derivative losses (\cref{multiplier lemma} replacing \cref{Coifman-Meyer}) built into global estimates. A dyadic approach could therefore significantly improve $N,k,s$, but this is not the point of the present work.
	\end{remark}
	
	Observe, that the initial data \eqref{initial data norm}, and small time existence (\cref{local existence}) yields that given $M$ sufficiently high, there's a solution on $I\supset[0,1]$ time interval. However, to get  \eqref{bootstrap assumption} to hold, we need weights on the initial data, which of course can be neglected if the data is supported in some finite ball. Thus, we introduce the norm on initial data that we will use in the global existence theorem:
	\begin{equation}\label{weighted initial data norm}
		\sob{u_0}{N}+M\sob{v_0}{N}+\sob{xu_0}{k+3/2}+M\sob{xv_0}{k+3/2}<E/M.
	\end{equation}
	
	Observe the discrepancy between powers of $M$ in the bootstrap ($M^{1.5}\norm{v}_{H^N}$) and initial data ($M^{2}\norm{v}_{H^N}$). Indeed, one may lower the power of $M$ for the heavy field in the initial data to $1.5$, but the bootstrap assumption cannot be strengthened up to the initial data power ($M^2$) given in \ref{weighted initial data norm} using our approach, since the estimates in \cref{energy} would not close.
	
	\subsection{Global solutions}
	To prove the main theorem, we will need a local existence result.
	\begin{theorem}[Local existence]\label{local existence}
		Let $u_0,v_0$ be initial data satisfying \cref{weighted initial data norm}, then for $M$ sufficiently large, there exist $u,v$ on time interval $I\supset[0,1]$ such that $u,v$ solve \cref{eom} with initial data $u_0,v_0$, moreover
		\begin{equation}\label{local bootstrap assumption}
			\norm{(u,v)}_{X(I)}<2\frac{E}{M}.
		\end{equation}
		
	\end{theorem}
	\begin{proof}
		In \cite{reall_effective_2022} Theorem 14, they prove an $M$ independent lower bound on time existence for similar equation, where their existence time $T$ only depends on
		\begin{equation*}
			\mathcal{E}=\sob{u_0}{N}+M\sob{v_0}{N}.
		\end{equation*}
		Moreover $T\sim\frac{1}{\mathcal{E}}$ in the $\mathcal{E}\to0$ limit. Their results extend to this problem. We may use this result with the fact that in \cref{weighted initial data norm} we have $\mathcal{E}=\frac{E}{M}$ to get existence on time interval $[0,1]$ for sufficiently high $M$. In fact $M\sim E$ suffices. We will show \cref{local bootstrap assumption} only for the light field energy norm, but all are bounded in a similar manner using the finiteness of the time interval. In particular, we don't need to show any amount of decay, as we're at a fixed time. We use unitarity of $e^{is\jpns{D}}$, triangle inequality, Minkowski inequality, Sobolev embedding and the bootstrap assumption \cref{local bootstrap assumption} to get
		\begin{equation}
			\sob{u}{N}=\sob{l}{N}\leq\sob{u_0}{N}+\int_0^1\d s \sob{u}{N}\sob{v}{N}\leq\frac{E}{M}+4\frac{E^2}{M^{2.5}}<1.5\frac{E}{M},
		\end{equation}
		where we used $M$ sufficiently large in the last step.
		This is a strengthening of the bootstrap, so it indeed holds by continuity.
	\end{proof}
	
	The main body of the paper will consist of proving the following 
	\begin{prop}[A priori bounds]\label{a priori bounds}
		Let $u,v$ be solutions to \cref{eom} on $[0,T]$, $T>1$, satisfying \cref{bootstrap assumption}. Then, for $M$ sufficiently large, 
		\begin{equation}
			\norm{(e^{i\jpns{D}t}L,e^{i\jpns{D}_Mt}H)}_{X([0,T])}\leq C\frac{\jpns{E}^3}{M^{1.5}}
		\end{equation}
		with $C$ independent of $T$ and $M$.
	\end{prop}
	
	Using this a priori bound, we close the bootstrap in the usual manner to obtain
	\begin{theorem}[Global existence]\label{global existence}
		Let $u_0,v_0$ be initial data satisfying \cref{weighted initial data norm}, $M$ sufficiently large (depending on the value of $E$), then there exist $u,v$ defined for all time such that $u,v$ solve \cref{eom} with initial data $u_0,v_0$. 
	\end{theorem}
	\begin{proof}
		Let's take $u,v$ first defined on some interval $I\supset[0,1]$ provided by \cref{local existence} (this already requires $M$ sufficiently large). Then, using \cref{a priori bounds} and the fact that $\norm{(u,v)}_{X(I)}<2\frac{E}{M}$ we may close the bootstrap and conclude that indeed $u,v$ satisfy \cref{bootstrap assumption}.
		
		Let's look at one component, all the rest follow in the same way. For ease of notation, set $1/p=1/6-\delta$, then, using the definition of $L$, triangle inequality, \cref{a priori bounds} and \cref{local existence} we get 
		\begin{equation}
			\jpns{t}^{3(\frac{1}{2}-\frac{1}{p})}\norm{u(t)}_{W^{k,p}}=\jpns{t}^{3(\frac{1}{2}-\frac{1}{p})}\norm{u(1)+e^{it\jpns{D}}L}_{W^{k,p}}\leq 2\frac{E}{M}+C\frac{\jpns{E}^3}{M^{1.5}}\leq 3\frac{E}{M} 
		\end{equation}
		where $C$ is the constant from \cref{a priori bounds} and note, that we used $M$ sufficiently big in the last inequality, as we know that $C$ doesn't depend on $M$.
		
		This strengthens \cref{bootstrap assumption} and by continuity we get that \cref{bootstrap assumption} holds for all times. By local existence, we know that the solution must be global.
	\end{proof}

	\section{Frequency analysis}\label{frequency analysis}
	\subsection{Separation of resonances}
	Throughout this section, we will fix $M$ sufficiently large so that it's bigger than any fixed numerical factor that it is compared to.
	
	Looking at the quadratic terms of \eqref{eom}, we find that we have to analyse the behaviour of the following two sets of frequencies
	\begin{align}
		\phi_u^{\epsilon_0\epsilon_1\epsilon_2}=\epsilon_0\jpns{\rho}+\epsilon_1\jpns{\nu}_M+\epsilon_2\jpns{\rho-\nu} \label{u frequency} \\
		\phi_v^{\epsilon_0\epsilon_1\epsilon_2}=\epsilon_0\jpns{\rho}_M+\epsilon_1\jpns{\nu}+\epsilon_2\jpns{\rho-\nu}, \label{v frequency}
	\end{align}
	where $\epsilon_i\in\{\pm1\}$. Let's start with analysing the spacetime resonance of \eqref{u frequency}, and drop the superscript for readability. We can immediately see, using the triangle inequality that the only case when \cref{u frequency} can vanish is if $\epsilon_0=\epsilon_2=-\epsilon_1$ (without loss of generality $\epsilon_0=1$), so we fix $\phi_u=\jpns{\rho}-\jpns{\nu}_M+\jpns{\rho-\nu}$ as other cases are strictly easier to treat. Therefore, we get
	\begin{equation}
		\begin{gathered}
			\nabla_\nu\phi_u=-\frac{\nu}{\jpns{\nu}_M}-\frac{\rho-\nu}{\jpns{\rho-\nu}}.
		\end{gathered}
	\end{equation}
	We have the following resonance sets
	\begin{itemize}
		\item time: $\phi_u=0$. For fixed $\rho$ this shape corresponds to an ellipsoid with major (minor) axis of size $m\sqrt{(1+\abs{\rho}^2)(m^2-4)}$ ($m\sqrt{m^2-4}$) and centre at $\hat{\rho}m^2\abs{\rho}/2$.
		\item space: $\abs{\nabla_\nu\phi_u}=0$ that implies $\nu=\rho\frac{M}{M-1}$
		\item spacetime: $\phi_u=\abs{\nabla_\nu\phi_u}=0$ which defines the empty set:
	\end{itemize}
	\begin{equation}
		\begin{gathered}
			\abs{\nabla_\nu\phi_u}=0\implies\abs{\nu}=M\abs{\rho-\nu}\implies \abs{\phi_u}=(M-1)\jpns{\rho-\nu}-\jpns{(M-1)(\rho-\nu)}\gtrsim \frac{M}{\jpns{\rho-\nu}}.
		\end{gathered}
	\end{equation}

	This also establishes our first result
	\begin{lemma}
		For $\phi_u$ defined above, we have
		\begin{equation}
			\begin{gathered}
				\abs{\phi_u}|_{\nabla_{\nu\phi_u=0}}\gtrsim\frac{M}{\jpns{\rho-\nu}}=\frac{M}{\jpns{\nu/M}}.
			\end{gathered}
		\end{equation}
	\end{lemma}
	\begin{proof}
	\end{proof}
	We also need to get a bound the other way around:
	\begin{lemma}
		For $\phi_u$ defined above, we have
		\begin{equation}
			\begin{gathered}
				\abs{\nabla_\nu\phi_u}|_{\phi_u=0}\gtrsim\max(\frac{1}{\jpns{\rho-\nu}^2},\frac{1}{\jpns{\nu/M}^2})
			\end{gathered}
		\end{equation}
	\end{lemma}
	\begin{proof}
		In order to get this, consider the following rearrangement
		\begin{equation}\label{nabla phi split}
			\begin{gathered}
				\abs{\nabla_\nu\phi_u}^2=\frac{\abs{\nu}^2}{\jpns{\nu}_M^2}+\frac{\abs{\rho-\nu}^2}{\jpns{\rho-\nu}^2}-\frac{2\nu\cdot(\rho-\nu)}{\jpns{\nu}_M\jpns{\rho-\nu}}\\=\Big(\frac{\abs{\nu}}{\jpns{\nu}_M}-\frac{\abs{\rho-\nu}}{\jpns{\rho-\nu}}\Big)^2+2(1+\cos\theta)\frac{\abs{\nu}\abs{\rho-\nu}}{\jpns{\nu}_M\jpns{\rho-\nu}}=:I^2+II^2,
			\end{gathered}
		\end{equation}
		with $\theta$ the angle between the incoming frequencies. We split $\abs{\nabla_\nu\phi_u}$ into the sum of two non negative quantitates, the first term vanishes iff $\abs{\nu}=M\abs{\rho-\nu}$ while the second if $\cos\theta=-1$, which leads us to consider the parametrisation $\abs{\nu}=M\lambda\abs{\rho-\nu}$, $\cos\theta=-1+\delta$ for $\delta,\lambda\in\R$. Solving a minimisation problem, one can indeed show that
		\begin{equation}\label{1st part}
			\begin{gathered}
				\abs{\frac{\lambda x}{\jpns{x\lambda}}-\frac{x}{\jpns{x}}}\gtrsim\begin{cases}
					\frac{x}{\jpns{x}\jpns{\lambda x}^2}\abs{1-\lambda^{2}} & \lambda<1 \\
					\frac{x\lambda}{\jpns{\lambda x}\jpns{x}^2}\abs{1-\lambda^{-2}} & \lambda>1
				\end{cases}
			\end{gathered}
		\end{equation}
		for $x(=\abs{\rho-\nu}),\lambda\in \R^+$. After a few lines of algebra, we get
		\begin{equation}\label{phi=0}
			\begin{gathered}
				\phi_u=0 \iff M^2/2=\jpns{\nu}_M\jpns{\rho-\nu}+\nu\cdot(\rho-\nu)
			\end{gathered}
		\end{equation}
		, which in the new parametrisation reads
		\begin{equation}\label{phi=0 in parametrisation}
			\begin{gathered}
				M/2=\jpns{\lambda x}\jpns{x}+(-1+\delta)x^2\lambda\leq 2\jpns{\lambda x}\jpns{x}
			\end{gathered}
		\end{equation}
		where $x=\abs{\rho-\nu}$. Now we have to do a case splitting
		\begin{itemize}
			\item $\sqrt{M}/2<\jpns{x}$: from \eqref{1st part}, we get that $\lambda\in[1-1/10,1+1/10]$, otherwise $\abs{\nabla_\nu\phi_u}\gtrsim\frac{1}{\jpns{x}^2},\frac{1}{\jpns{\lambda x}^2}$. Therefore, we can expand the expression in \eqref{phi=0 in parametrisation} to get that
			\begin{equation}\label{bound on angle}
				\begin{gathered}
					\delta=\frac{M-1/\lambda-\lambda}{2\lambda x^2}+O(\frac{1}{x^3})
				\end{gathered}
			\end{equation}
			\item $\sqrt{M}/2<\jpns{\lambda x}$: we use the second inequality in \eqref{1st part} to get $\lambda\sim 1$, so we again conclude \eqref{bound on angle}.
		\end{itemize}
		\ \\
		\cref{bound on angle} in the range $\lambda\sim1$ implies
		\begin{equation}
			\begin{gathered}
				II^2=\delta\frac{x^2\lambda}{\jpns{\lambda x}\jpns{x}}\gtrsim \frac{1}{x^2}.
			\end{gathered}
		\end{equation}
		Therefore, we get $\abs{\nabla_\nu\phi_u}\gtrsim\max(\frac{1}{\jpns{x}^2},\frac{1}{\jpns{\lambda x}^2})$, which is what we wanted.
		
	\end{proof}
	The next problem is to  realise that not only the resonance sets are disjoint, but that their separation does not degenerate either with $M$ or with the increase of frequency:
	\begin{lemma}
		For $\phi_u$ as above, we have
		\begin{equation}
			\begin{gathered}
				\text{dist}(\{\phi_u=0\},\{\nabla_\nu\phi_u=0\})\gtrsim \sqrt{M}
			\end{gathered}
		\end{equation}
	\end{lemma}
	\begin{proof}
		We know that $\nabla_\nu\phi_u=0$ corresponds to the point $\nu_0=\rho\frac{M}{M-1}$, so we just need to add an arbitrary perturbation to this and find when $\phi_u=0$ can be satisfied. Let's set $\nu=\nu_0+\eta$ and plug this into \eqref{phi=0}
		\begin{equation}
			\begin{gathered}
				(M/2-(\rho-\nu)\cdot\nu/M)^2=\Big(M/2+\frac{\abs{\rho}^2}{(M-1)^2}+\frac{\rho\cdot\eta(1+1/M)}{M-1}+\frac{\abs{\eta}^2}{M^2}\Big)^2\\=\jpns{\nu/M}^2\jpns{\rho-\nu}^2=\Big(\abs{\frac{\rho}{M-1}+\frac{\eta}{M}}^2+1\Big)\Big(\abs{\frac{\rho}{M-1}+\eta}^2+1\Big).
			\end{gathered}
		\end{equation}
		Expending the brackets and cancelling some of the terms yields
		\begin{equation}
			\begin{gathered}
				\abs{\eta}^2(1-\frac{1}{M})(\frac{\abs{\rho}^2}{(M-1)^2}+1)-\frac{M-2}{M-1}\rho\cdot\eta(1+1/M)\\=(M^2/4-1)+\frac{\abs{\rho}^2(M-2)}{(M-1)^2}+\frac{(\rho\cdot\eta)^2(1-1/M)}{(M-1)^2}
			\end{gathered}
		\end{equation}
		with both sides containing only non-negative terms, except the dot product on the left hand side. For equality to hold, we know that either the first or the second term on LHS must be greater than half of the RHS. In the former, we get
		\begin{equation}
			\begin{gathered}
				\abs{\eta}^2(\frac{\abs\rho^2}{M^2}+1)\gtrsim M^2+\abs{\rho}^2/M\\
				\implies \abs{\eta}\gtrsim \sqrt{M},
			\end{gathered}
		\end{equation}
		while in the latter
		\begin{equation}
			\begin{gathered}
				\abs{\rho}\abs{\eta}\gtrsim {M^2}+1)\gtrsim M^2+\abs{\rho}^2/M \gtrsim M^2+\sqrt{M}\abs{\rho}+\abs{\rho}^2/M\\
				\implies \abs{\eta}\gtrsim\sqrt{M}.
			\end{gathered}
		\end{equation}
		
	\end{proof}
	
	Now that we've established lower bounds and separation on the resonance sets, it is time to get a partition of unity:
	\begin{lemma}\label{chi lemma}
		There exist $\chi_S,\chi_T:\R^{6}\to[0,1]$, such that $\chi_S|_{\{\phi_u=0\}}=\chi_T|_{\{\nabla_\nu\phi_u=0\}}=0$, $\chi_S+\chi_T=1$ and 
		\begin{equation}
			\begin{gathered}
				\abs{\phi_u}|_{\supp \chi_S}\gtrsim\frac{M}{\jpns{\rho-\nu}}\gtrsim\frac{M}{\jpns{\nu/M}}\\
				\abs{\nabla_\nu\phi_u}|_{\supp \chi_T}\gtrsim\max(\frac{1}{\jpns{\rho-\nu}^2},\frac{1}{\jpns{\nu/M}^2}).
			\end{gathered}
		\end{equation}
		In particular, we may choose $\chi_S$ to have support on a rectangle with side lengths 2 and $\jpns{\frac{\rho}{M}}$.
	\end{lemma}
	\begin{proof}
		Let $\chi:\R\to[0,1]$ be a smooth cut off function such that 
		\begin{equation}
			\begin{gathered}
				\chi(x)=\begin{cases}
					1 & \abs{x}\leq 1\\
					0 & \abs{x}\geq 2
				\end{cases}
			\end{gathered}
		\end{equation}
		and smoothly interpolating between the two. Then, let 
		\begin{equation}
			\begin{gathered}
				\chi_S(\rho,\nu)=\chi(\abs{\nu}/2)\chi(\abs{\rho})+(1-\chi(\abs{\rho}))\chi\big(\abs{\nu^\perp})\chi(\abs{\rho\frac{M}{M-1}-\nu^\parallel}\frac{4}{\jpns{\frac{\rho}{M}}}\big)
			\end{gathered}
		\end{equation}
		where $\nu^\perp,\nu^\parallel$ are the orthogonal and parallel projections of $\nu$ along $\rho$. The term $4/\jpns{\rho/M}$ is there, so that at high frequencies the width along the parallel direction is about $\abs\rho/4M$, while the separation for small and large $\rho$ ensures smoothness.
		
		$\chi_S$ is clearly smooth, so we just need to verify the bounds. We start with $\nabla_\nu\phi_u$ and we will use the form \eqref{nabla phi split} and case splitting. Before doing so, note that $I$ vanishes only when $\abs{\nu}=M\abs{\rho-\nu}$, which happens (for fixed $\rho$) on a sphere centred at $\frac{M^2\rho}{M^2-1}(\approx\rho)$ of radius $\abs{\rho}\frac{M}{\sqrt{M^2-1}}(\approx\abs{\rho}/M)$. When away from this region, we can give an explicit bound for its size as 
		\begin{subnumcases}{I\gtrsim}
			\frac{1}{\jpns{\nu/M}^2}\gtrsim \frac{1}{\jpns{\rho-\nu}^2} & $\abs{\nu-\rho}>\max(2\frac{\abs{\rho}}{M},1/8),\abs{\rho}>1$ \label{out of circle}
			\\
			\frac{1}{\jpns{\rho-\nu}^2} \gtrsim \frac{1}{\jpns{\nu/M}^2}& $\abs{\nu-\rho}<\frac{1}{2}\frac{\abs{\rho}}{M},\abs{\rho}>M$ \label{in circle}
			\\
			\frac{1}{\jpns{\nu/M}^2}\gtrsim \frac{1}{\jpns{\rho-\nu}^2} & $\abs{\nu}>2,\abs{\rho}<1$ \label{very small circle}
		\end{subnumcases}
		which may be found using expansion of the square roots.
		
		Similarly, using $1+\cos\theta\gtrsim 1$ for $\abs{\nu^\perp}>1$ we have a bound $II$ away from the set $\{\cos\theta=-1\}$
		\begin{subnumcases}{II\gtrsim}
			\frac{1}{\jpns{\rho-\nu}^2} & $\abs{\nu^\perp}\in[1,2\abs{\rho}/M],\abs{\rho}>M/8$ \label{costheta perpendicular}
			\\
			1 & $\abs{\rho}>M/8,\abs{\rho-\nu}\in\frac{\abs{\rho}}{M}[1/2,2],\cos\theta\geq0$ \label{costheta good side}
		\end{subnumcases}
		
		See Figure \cref{circles}.
		
		\begin{figure}[t]\label{circles}
			\centering
			\def\svgwidth{\columnwidth}
			\scalebox{0.5}{
\begingroup%
  \makeatletter%
  \providecommand\color[2][]{%
    \errmessage{(Inkscape) Color is used for the text in Inkscape, but the package 'color.sty' is not loaded}%
    \renewcommand\color[2][]{}%
  }%
  \providecommand\transparent[1]{%
    \errmessage{(Inkscape) Transparency is used (non-zero) for the text in Inkscape, but the package 'transparent.sty' is not loaded}%
    \renewcommand\transparent[1]{}%
  }%
  \providecommand\rotatebox[2]{#2}%
  \newcommand*\fsize{\dimexpr\f@size pt\relax}%
  \newcommand*\lineheight[1]{\fontsize{\fsize}{#1\fsize}\selectfont}%
  \ifx\svgwidth\undefined%
    \setlength{\unitlength}{381.06074524bp}%
    \ifx\svgscale\undefined%
      \relax%
    \else%
      \setlength{\unitlength}{\unitlength * \real{\svgscale}}%
    \fi%
  \else%
    \setlength{\unitlength}{\svgwidth}%
  \fi%
  \global\let\svgwidth\undefined%
  \global\let\svgscale\undefined%
  \makeatother%
  \begin{picture}(1,1.06872724)%
    \lineheight{1}%
    \setlength\tabcolsep{0pt}%
    \put(0,0){\includegraphics[width=\unitlength,page=1]{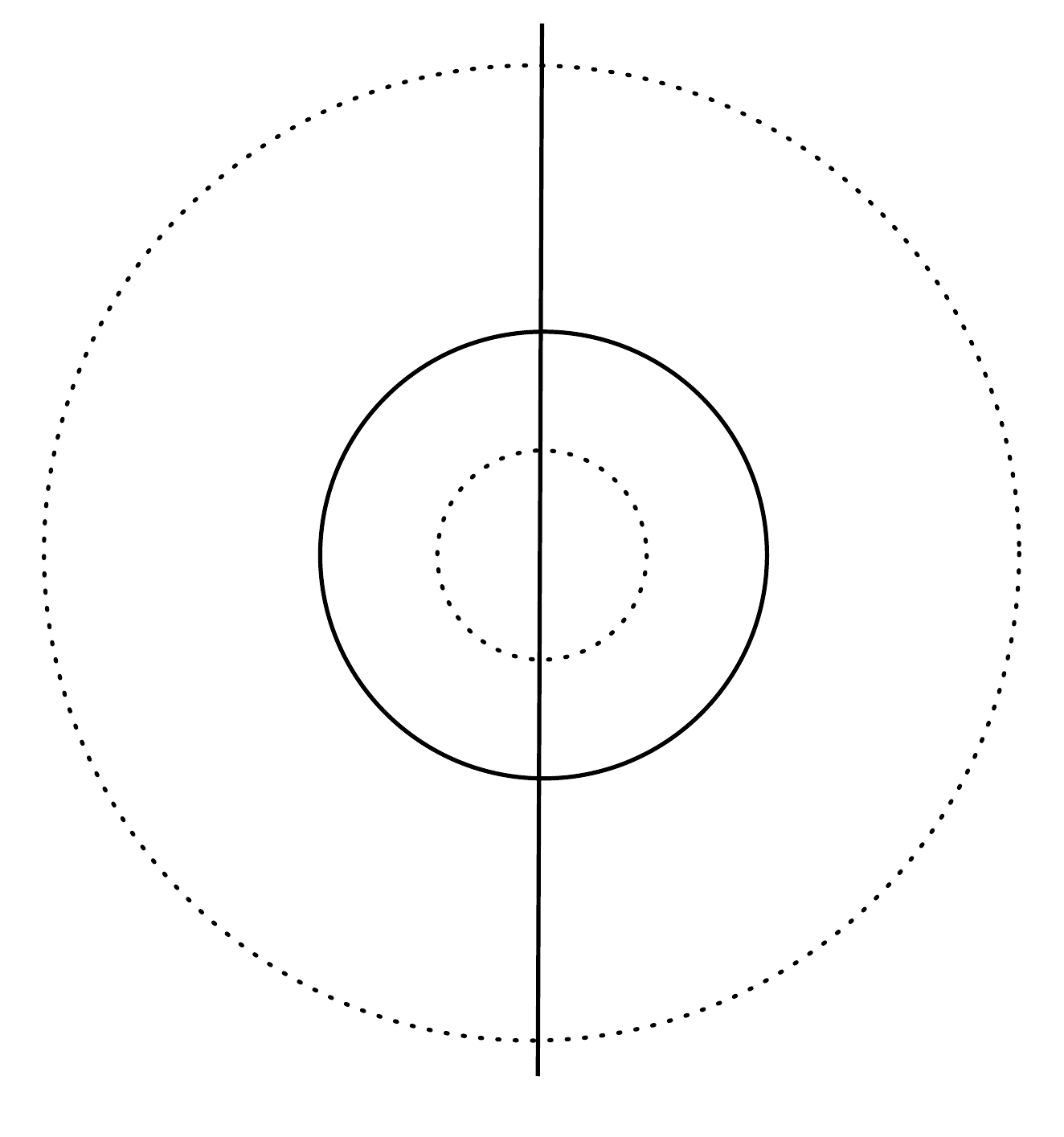}}%
    \put(0.54714798,0.88122017){\makebox(0,0)[lt]{\lineheight{1.25}\smash{\begin{tabular}[t]{l}$\cos\theta<0$\end{tabular}}}}%
    \put(0.31894852,0.88105218){\makebox(0,0)[lt]{\lineheight{1.25}\smash{\begin{tabular}[t]{l}$\cos\theta>0$\end{tabular}}}}%
    \put(0,0){\includegraphics[width=\unitlength,page=2]{circles.pdf}}%
    \put(0.51294155,0.50566886){\makebox(0,0)[lt]{\lineheight{1.25}\smash{\begin{tabular}[t]{l}$\nu=\rho$\end{tabular}}}}%
    \put(0.65333499,0.36472102){\makebox(0,0)[lt]{\lineheight{1.25}\smash{\begin{tabular}[t]{l}$I=0$\end{tabular}}}}%
    \put(0.76507747,0.55773069){\makebox(0,0)[lt]{\lineheight{1.25}\smash{\begin{tabular}[t]{l}$\cos\theta=-1$\end{tabular}}}}%
    \put(0,0){\includegraphics[width=\unitlength,page=3]{circles.pdf}}%
    \put(0.38865192,0.68641794){\makebox(0,0)[lt]{\lineheight{1.25}\smash{\begin{tabular}[t]{l}$|\rho|/M$\end{tabular}}}}%
  \end{picture}%
\endgroup%
}
			\caption{For fixed $\rho$, the behaviour of $\nabla_\nu\phi_u$ as a function of $\eta$, following the discussion of \cref{chi lemma} .The cylindrical symmetry is used for the plot}
		\end{figure}
		
		\begin{figure}
			\centering
			\begin{subfigure}[b]{0.3\textwidth}
				\centering
				\scalebox{0.4}{
\begingroup%
  \makeatletter%
  \providecommand\color[2][]{%
    \errmessage{(Inkscape) Color is used for the text in Inkscape, but the package 'color.sty' is not loaded}%
    \renewcommand\color[2][]{}%
  }%
  \providecommand\transparent[1]{%
    \errmessage{(Inkscape) Transparency is used (non-zero) for the text in Inkscape, but the package 'transparent.sty' is not loaded}%
    \renewcommand\transparent[1]{}%
  }%
  \providecommand\rotatebox[2]{#2}%
  \newcommand*\fsize{\dimexpr\f@size pt\relax}%
  \newcommand*\lineheight[1]{\fontsize{\fsize}{#1\fsize}\selectfont}%
  \ifx\svgwidth\undefined%
    \setlength{\unitlength}{531bp}%
    \ifx\svgscale\undefined%
      \relax%
    \else%
      \setlength{\unitlength}{\unitlength * \real{\svgscale}}%
    \fi%
  \else%
    \setlength{\unitlength}{\svgwidth}%
  \fi%
  \global\let\svgwidth\undefined%
  \global\let\svgscale\undefined%
  \makeatother%
  \begin{picture}(1,0.38418079)%
    \lineheight{1}%
    \setlength\tabcolsep{0pt}%
    \put(0,0){\includegraphics[width=\unitlength,page=1]{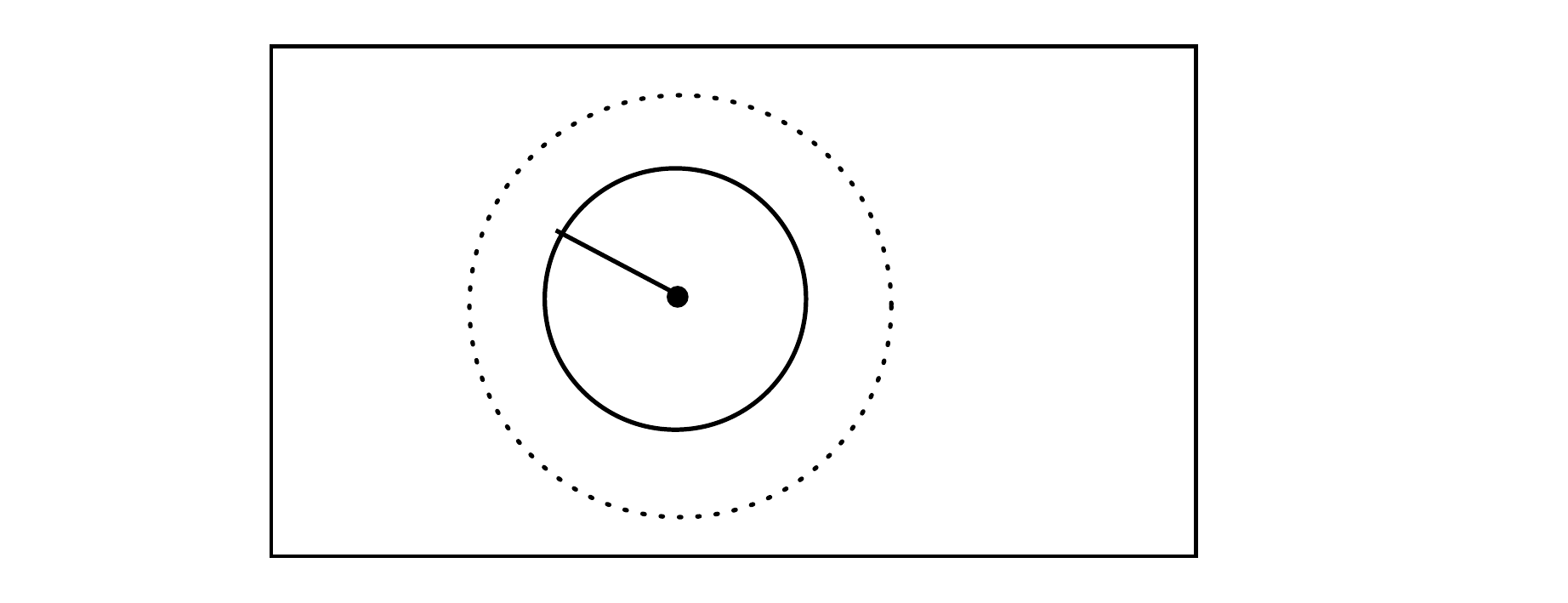}}%
    \put(0.44067797,0.18502825){\makebox(0,0)[lt]{\lineheight{1.25}\smash{\begin{tabular}[t]{l}$\rho$\end{tabular}}}}%
    \put(0.37288136,0.23022599){\makebox(0,0)[lt]{\lineheight{1.25}\smash{\begin{tabular}[t]{l}$|\rho|/M$\end{tabular}}}}%
  \end{picture}%
\endgroup%
}
				\caption{$\abs{\rho}<M/4$}
				\label{small rho}
			\end{subfigure}
			\hfil
			\begin{subfigure}[b]{0.3\textwidth}
				\centering
				\scalebox{0.4}{
\begingroup%
  \makeatletter%
  \providecommand\color[2][]{%
    \errmessage{(Inkscape) Color is used for the text in Inkscape, but the package 'color.sty' is not loaded}%
    \renewcommand\color[2][]{}%
  }%
  \providecommand\transparent[1]{%
    \errmessage{(Inkscape) Transparency is used (non-zero) for the text in Inkscape, but the package 'transparent.sty' is not loaded}%
    \renewcommand\transparent[1]{}%
  }%
  \providecommand\rotatebox[2]{#2}%
  \newcommand*\fsize{\dimexpr\f@size pt\relax}%
  \newcommand*\lineheight[1]{\fontsize{\fsize}{#1\fsize}\selectfont}%
  \ifx\svgwidth\undefined%
    \setlength{\unitlength}{400.5bp}%
    \ifx\svgscale\undefined%
      \relax%
    \else%
      \setlength{\unitlength}{\unitlength * \real{\svgscale}}%
    \fi%
  \else%
    \setlength{\unitlength}{\svgwidth}%
  \fi%
  \global\let\svgwidth\undefined%
  \global\let\svgscale\undefined%
  \makeatother%
  \begin{picture}(1,1)%
    \lineheight{1}%
    \setlength\tabcolsep{0pt}%
    \put(0,0){\includegraphics[width=\unitlength,page=1]{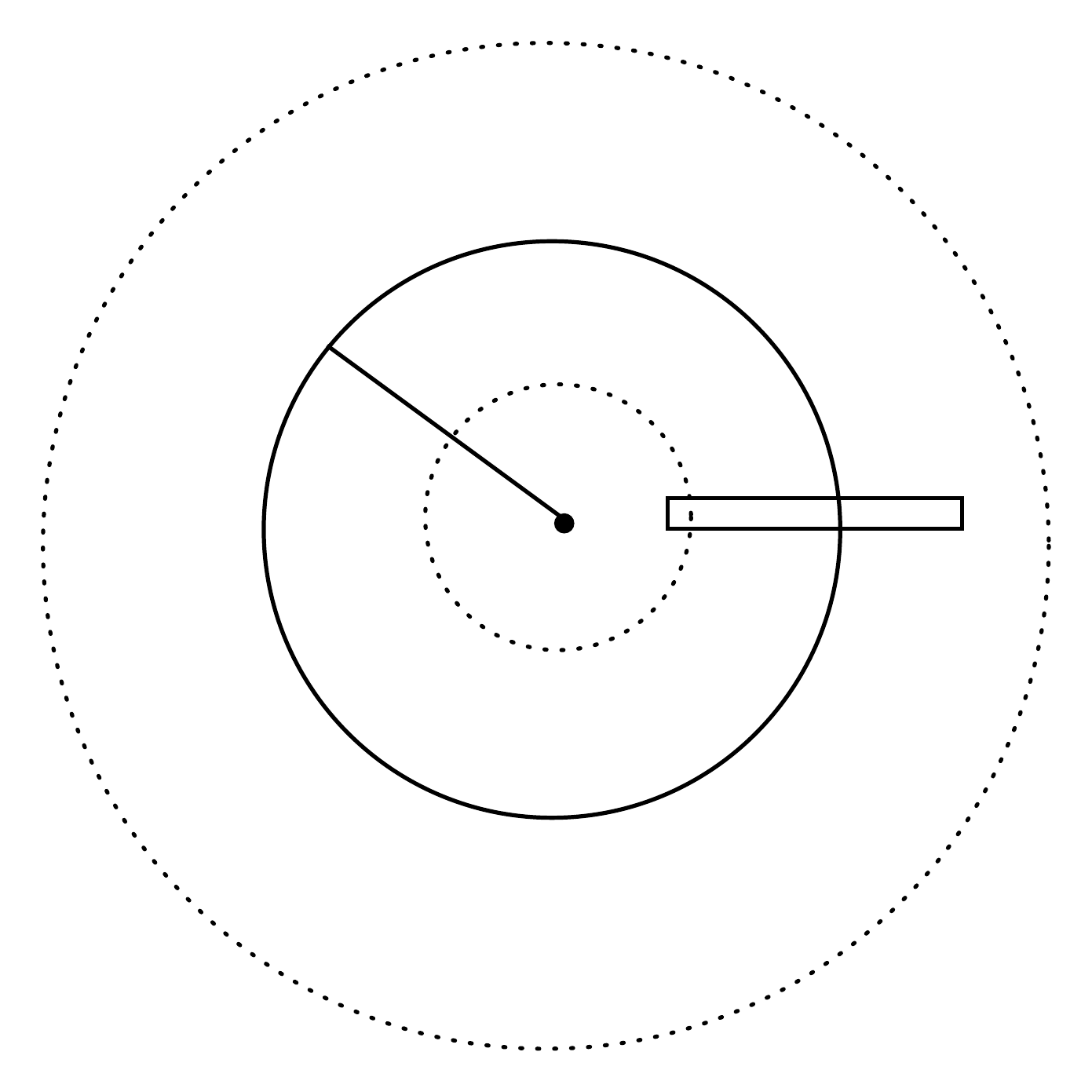}}%
    \put(0.50561798,0.47565543){\makebox(0,0)[lt]{\lineheight{1.25}\smash{\begin{tabular}[t]{l}$\rho$\end{tabular}}}}%
    \put(0.3576779,0.67602996){\makebox(0,0)[lt]{\lineheight{1.25}\smash{\begin{tabular}[t]{l}$|\rho|/M$\end{tabular}}}}%
    \put(0.78651685,0.55805243){\makebox(0,0)[lt]{\lineheight{1.25}\smash{\begin{tabular}[t]{l}$|\rho|/2M$\end{tabular}}}}%
  \end{picture}%
\endgroup%
}
				\caption{$\abs{\rho}>M$}
				\label{large rho}
			\end{subfigure}
			\caption{Change of resonance regions (circles) for two different value of $\rho$ and support of $\chi_T$ (rectangle)}
			\label{circle rectangle intersections}
		\end{figure}
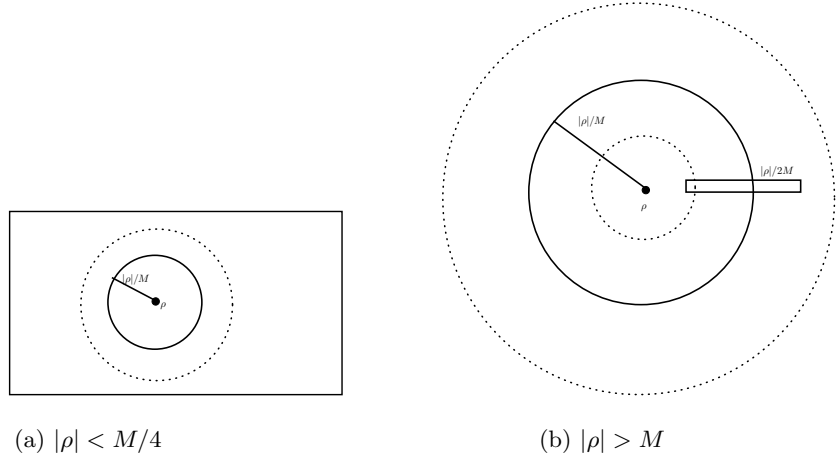

		Therefore, the case splitting will happen according to the relative position of the sphere ($\nabla_\nu\phi_u=0$) and $\supp\chi_S$.
		\begin{itemize}
			\item $\frac{\abs{\rho}}{M}\leq\frac{1}{16}$ see \cref{small rho}. In this case, the circle of double radius (ie. $\abs{\nu}=\abs{\rho-\nu}M/2$) is not contained in $\supp\chi_T$ (see \cref{small rho}) and we also have $\abs{\rho-\nu}>1/8$, therefore, we can use the bounds \eqref{out of circle} and \eqref{very small circle}. This gives $I\gtrsim\frac{1}{\jpns{\nu/M}^2}$.
			\item $\frac{\abs{\rho}}{M}>\frac{1}{8}$ see \cref{large rho}. Now, we have to treat 3 parts of the problem separately. Going from the outer to inner, we can use \eqref{out of circle} for $\abs{\nu-\rho}>2\frac{\abs\rho}{M}$ to yield $II\gtrsim\frac{1}{\jpns{\abs{\nu}/M}^2}$. In the region $\abs{\nu-\rho}\in\frac{\abs{\rho}}{M}[1/2,2]$, we either have favourable sine of the cosine or $\abs{\nu^\perp}>1$, thus we may conclude $II\gtrsim\frac{1}{\jpns{\rho-\nu}^2}\sim\frac{1}{\jpns{\nu/M}^2}$ via \eqref{costheta perpendicular}-\eqref{costheta good side}. In the inner most part $\abs{\nu-\rho}<\frac{\abs{\rho}}{2M}$, we use \eqref{in circle} to get $I>\frac{1}{\jpns{\rho-\nu}^2}$. 
		\end{itemize}
		
		Therefore, we may conclude 
		\begin{equation}
			\begin{gathered}
				\abs{\nabla_\nu\phi_u}|_{\supp\chi_T}\gtrsim\max(\frac{1}{\jpns{\rho-\nu}^2},\frac{1}{\jpns{\nu/M}^2}).
			\end{gathered}
		\end{equation}
		
		Now we turn our attention to $\abs{\phi_u}|_{\supp\chi_S}$. First, note that for $\abs{\rho}<M/4$ on $\supp\chi_S$, we have $\abs{\rho-\nu}<10$, so $\abs{\phi_u}>M/2$. Hence, wlog take $\abs{\rho}>M/4$, so that we can write $\nu=\rho\alpha+\nu^\perp$ for $\abs{\nu^\perp}<1,\alpha\in\frac{M}{M-1}+\jpns{\frac{\rho}{M}}[-2,2]$. Then, we can use binomial expansion and a few lines of algebra shows
		\begin{equation}
			\begin{gathered}
				\phi_u\gtrsim\frac{M^2}{\jpns{\rho}}\sim\frac{M}{\jpns{\rho-\nu}}.
			\end{gathered}
		\end{equation}
	\end{proof}
	
	Having dealt with $\phi_u$ above, we now move on to $\phi_v$, where the only vanishing set will be provided by $\phi_v=-\jpns{\rho}_M+\jpns{\rho-\nu}+\jpns{\nu}$. This is related to the frequency analysed above via the change $\rho\leftrightarrow\nu$, but now, the other quantity of interest is $\nabla_\rho\phi_u=\frac{\rho}{\jpns{\rho}}+\frac{\rho-\nu}{\jpns{\rho-\nu}}$ which vanishes at $\nu=2\rho$. This point lies much further to the inside of the ellipse $\phi_u=0$ 
	than the one analysed above, so using similar arguments as above, and noting that neither incoming frequency is weighted by $M$, we get
	\begin{lemma}
	There exist $\chi_S,\chi_T:\R^{6}\to[0,1]$, such that $\chi_S|_{\{\phi_v=0\}}=\chi_T|_{\{\nabla_\nu\phi_v=0\}}=0$, $\chi_S+\chi_T=1$ and 
	\begin{equation}
		\begin{gathered}
			\abs{\phi_v}|_{\supp \chi_S}\gtrsim\frac{M}{\jpns{\rho-\nu}}\gtrsim\frac{M}{\jpns{\nu}}\\
			\abs{\nabla_\nu\phi_v}|_{\supp \chi_T}\gtrsim\max(\frac{1}{\jpns{\rho-\nu}^2},\frac{1}{\jpns{\nu}^2}).
		\end{gathered}
	\end{equation}
	In particular, we may choose $\chi_S$ to have support on a rectangle with side lengths 2 and $\jpns{\rho}$.
	\end{lemma}
	
	\subsection{Symbolic estimates}
	To have paraproduct estimates, it's not sufficient to have an $L^\infty$ bound on the multipliers (eg. $\chi_S/\phi$), but we also need symbolic bounds, ie. $L^\infty$ estimates on derivative ($(\nu\partial_\nu)^s\chi_S/\phi$). In particular, the derivatives are performed with respect to the incoming frequencies, so for this section and the next, we'll use notation $\nu_1:=\nu$, $\nu_2=\rho-\nu$ and $\partial_{\nu_{1}}:=\partial_{\nu_1}|_{\nu_2},\partial_{\nu_{2}}:=\partial_{\nu_2}|_{\nu_1}$. Let's start with the phases.

	\begin{lemma}
		On $\supp\chi_S$ we have 
		\begin{align}
			\abs{\partial^\alpha_{\nu_1}\frac{1}{\phi_u}}\lesssim_\alpha\jpns{\nu_2}^{2\abs{\alpha}+1}/\jpns{\nu_1}^{\abs{\alpha}}\label{phi nu1 symbolic estimate}\\
			\abs{\partial^\alpha_{\nu_2}\frac{1}{\phi_u}}\lesssim_\alpha\frac{\jpns{\nu_2}^{\abs{\alpha}+1}}{M^{\abs{\alpha}+1}}\label{phi nu2 symbolic estimate}\\
		\end{align}
		and on $\supp\chi_T$ we have
		\begin{equation*}
			\begin{aligned}
				\abs{\partial^\alpha_{\nu_k}\frac{1}{\partial_{\nu_1}|_{\nu_1+\nu_2}\phi_u}}\lesssim_\alpha\frac{\min(\jpns{\nu_2},\jpns{\nu_1/M})^{2\abs{\alpha}}}{\jpns{\nu_k}^{\abs{\alpha}}}\\
			\end{aligned}
		\end{equation*}
	\end{lemma}
\begin{proof}
	First, note that $\jpns{x}\in S^1$ (symbol class 1), therefore $(\partial_x)^\alpha\jpns{x+v}\lesssim_\alpha\jpns{x}^{1-\abs{\alpha}}$ for any $v$ such that $\abs{x+v}\gtrsim\abs{x}$. From these, it follows that	$\abs{\partial_{\nu_1}^\alpha\phi_u}\lesssim_\alpha\jpns{\nu_1}^{1-\abs{\alpha}}$ and $\abs{\partial_{\nu_2}^\alpha\phi_u}\lesssim_\alpha\jpns{\nu_2}^{1-\abs{\alpha}}$ on the required support. 
	
	Proof of first statement. We expand the left hand side of \cref{phi nu1 symbolic estimate} in the region $\abs{\nu_1}>M$ and use \cref{chi lemma} with the above claims to get
	\begin{equation*}
		\begin{aligned}
			\abs{\partial_{\nu_1}^\alpha\frac{1}{\phi}}\lesssim_\alpha\sum_{\alpha_1+\alpha_2+...+\alpha_l=\alpha}\abs{\frac{1}{\phi^{l+1}}(\partial_{\nu_1}^{\alpha_1}\phi)...(\partial_{\nu_1}^{\alpha_l}\phi)} \lesssim\sum_{\alpha_1+\alpha_2+...+\alpha_l=\alpha}\frac{\jpns{\nu_2}^{l+1}}{M^{l+1}}\jpns{\nu_1}^{l-\abs{\alpha}}\lesssim \frac{\jpns{\nu_2}^{2\abs{\alpha}+1}}{\jpns{\nu_1}^{\abs{\alpha}}}
			\end{aligned}
		\end{equation*}
	where we used that on the required support $\abs{\nu_2}\sim\abs{\nu_1}$.
	
	Proof of \cref{phi nu2 symbolic estimate}: 
		
		\begin{equation*}
			\begin{aligned}
				\abs{\partial_{\nu_2}^\alpha\frac{1}{\phi}}\lesssim_\alpha\sum_{\alpha_1+\alpha_2+...+\alpha_l=\alpha}\abs{\frac{1}{\phi^{l+1}}(\partial_{\nu_2}^{\alpha_1}\phi)...(\partial_{\nu_2}^{\alpha_l})\phi} \lesssim\sum_{\alpha_1+\alpha_2+...+\alpha_l=\alpha}\frac{\jpns{\nu_2}^{l+1}}{M^{l+1}}\jpns{\nu_2}^{-\abs{\alpha}+l}\lesssim \frac{\jpns{\nu_2}^{\abs{\alpha}+1}}{M^{\abs{\alpha}+1}}
				\end{aligned}
			\end{equation*}
			
			Note that $I=\partial_{\nu_1}|_{\nu_1+\nu_2}\phi_u=\frac{\nu_1}{\jpns{\nu_1}_M}+\frac{\nu_2}{\jpns{\nu_2}}$, so using that $\frac{x}{\jpns{x}_M}$ is symbolic we get $\abs{\partial_{\nu_k}^\alpha I}\lesssim_\alpha\frac{1}{\jpns{\nu_k}^{\abs{\alpha}}}$. Therefore, we can compute as before to get the claim.
			
		\end{proof}

	\begin{remark}
		The above lemma is \textit{far} from optimal, but is sufficient in the sense that the loss is only on the lower frequency term.
	\end{remark}
	
	\begin{lemma}
		\begin{equation*}
			\begin{aligned}
				\abs{\partial_{\nu_2}^{\alpha}\partial_{\nu_1}^\beta\chi_S}\lesssim_{\alpha,\beta}(\frac{\abs{\nu_2}}{\abs{\nu_1}})^{\abs{\beta}}\lesssim1
			\end{aligned}
		\end{equation*}
		for all multi indices $\alpha,\beta$.
	\end{lemma}
	\begin{remark}
		Note, that $\nu_1$ is always the higher frequency, and we indeed gain $\nu_1$ factors with $\partial_{\nu_1}$, though at the cost of $\nu_2$.
	\end{remark}

	\begin{proof}
		To prove the lemma, we will use the symbolic bounds \begin{equation}\label{easy bound}
			(x\partial_x)^\alpha\frac{x^p}{\abs{x+y}^q}\lesssim_{p,q,\alpha}x^{p-q-\alpha}
		\end{equation} for $\abs{x}\sim\abs{x+y}\gtrsim1$.
		
		Clearly the above holds for $\alpha=\beta=0$, since $\chi_S$ is a cut-off function. For $\abs{\alpha}+\abs{\beta}>0$, we only need to prove the above claim on $\supp\text{d}\chi_S$, hence all the estimates below are to be understood on this set. Since for $\chi_S$ is smooth, it's sufficient to study the above estimates outside a compact set eg. $\{\abs{\nu_1},\abs{\nu_2}<100\}$. On this set, 		
		$\chi_S=\chi(I_0)\chi(I_1)$ for some cut-off function $\chi$ and 
		\begin{equation*}
			\begin{gathered}
				I_0:=(\frac{\nu_1^\parallel}{M-1}-\nu_2\frac{M}{M-1})/\jpns{\frac{\nu_1+\nu_2}{M}}\\
				I_1:=\nu_1\frac{\nu_1\cdot\nu_2+\abs{\nu_2}^2}{\abs{\nu_1+\nu_2}^2}-\nu_2\frac{\nu_1\cdot(\nu_1+\nu_2)}{\abs{\nu_1+\nu_2}^2}.
			\end{gathered}
		\end{equation*}
		Therefore, on $\supp\chi_S$ we have $\abs{\nu_1}\sim\abs{\nu_1+\nu_2}$ and $\abs{\nu_2}\sim\max(1,\abs{\nu_1}/M)$. Under these restrictions, using \cref{easy bound}, it follows that
		\begin{equation}
			\begin{gathered}
				\abs{\partial_{\nu_{1}}^\alpha\partial_{\nu_{1}}^\beta\nu_1^\parallel}\lesssim_{\alpha,\beta}\abs{\nu_1}^{\min(1-\abs{\alpha},0)}\abs{\nu_2}^{-\beta}\\
				\abs{\partial_{\nu_{1}}^\alpha\partial_{\nu_{2}}^\beta\nu_1^\perp}\lesssim_{\alpha,\beta}\abs{\nu_1}^{-\abs{\alpha}}
			\end{gathered}
		\end{equation}
		
		In turn, these imply
		\begin{equation*}
			\begin{aligned}
				\abs{(\partial_{\nu_2})^{\alpha}I_k} \lesssim_{\alpha}1\\
				\abs{(\partial_{\nu_1})^\alpha I_k}\lesssim_{\alpha}M^{-1}\abs{\nu_1}^{\min(0,1-\abs{\alpha})}
			\end{aligned}
		\end{equation*}
		for $k=0,1$. As we are interested in changes of the absolute value of $I_k$ on the region where $\abs{I_k}\sim1$, we use $\partial \abs{I}=\frac{I\cdot\partial I}{\abs{I}}$ to conclude 
		\begin{equation*}
			\begin{gathered}
				\abs{(\partial_{\nu_2})^\alpha \abs{I_k}}\lesssim_{\alpha}\sum_{\abs{\alpha_1}+2\abs{\alpha_2}+...+l\abs{\alpha_l}=\abs{\alpha}}\abs{(\partial_{\nu_2})^{\alpha_1}I_{k}}\abs{(\partial_{\nu_2})^{\alpha_2}I_{k}}...\abs{(\partial_{\nu_2})^{\alpha_l}I_{k}}\lesssim_{\alpha} 1\\
				\abs{(\partial_{\nu_1})^\alpha \abs{I_k}}\lesssim_{\alpha}M^{-\alpha}
			\end{gathered}
		\end{equation*}
		Using the fact that $\chi\in\mathcal{C}^\infty_c$, the lemma follows.
	\end{proof}
	
	Putting things together, we've proved
	\begin{lemma}\label{symbolic estimates}
		\begin{equation*}
			\begin{aligned}
				\abs{\partial_{\nu_k}^\alpha\frac{\chi_S}{\phi}}\lesssim_\alpha\min(\jpns{\nu_1},\jpns{\nu_2})^{2\abs{\alpha}+1}/\jpns{\nu_k}^{\abs{\alpha}}\\
				\abs{\partial_{\nu_k}^\alpha\frac{\chi_S}{\partial_{\nu_1}|_{\nu_1+\nu_2}\phi}}\lesssim_\alpha\frac{\min(\jpns{\nu_1},\jpns{\nu_2})^{2\abs{\alpha}}}{\jpns{\nu_k}^{\abs{\alpha}}}
			\end{aligned}
		\end{equation*}		
	\end{lemma}
	
	\subsection{Operators}
	
	Given all the estimates above, we can use Coifman-Meyer theory to conclude
	
	\begin{lemma}\label{multiplier lemma}
		For $\chi_S,\chi_T$ defined in \cref{chi lemma}, we have 
		\begin{equation}
			\begin{gathered}
				\norm{T_{\frac{\chi_S}{\phi}}(f,g)}_{W^{k,r}}\lesssim_{k,r,p,q,\bar{p},\bar{q}}\norm{f}_{W^{\mathfrak{a},p}}\norm{g}_{W^{k,q}}+\norm{f}_{W^{k,\bar{p}}}\norm{g}_{W^{\mathfrak{a},\bar{q}}}\\
				\norm{T_{\frac{\chi_T\nabla_\nu\phi}{\abs{\nabla_\nu\phi}^2}{\phi}}(f,g)}_{W^{k,r}}\lesssim_{k,r,p,q,\bar{p},\bar{q}}\norm{f}_{W^{\mathfrak{a},p}}\norm{g}_{W^{k,q}}+\norm{f}_{W^{k,\bar{p}}}\norm{g}_{W^{\mathfrak{a},\bar{q}}}
			\end{gathered}
		\end{equation}
		for any $f,g\in\mathcal{S}$, $k\geq \mathfrak{a}$, $\mathfrak{a}$ depending only on the multiplier and $\frac{1}{r}=\frac{1}{p}+\frac{1}{q}=\frac{1}{\bar{p}}+\frac{1}{\bar{q}}$. In fact, we may choose $\mathfrak{a}=40$ in the second.
	\end{lemma}

	\begin{remark}
		The choice of $\mathfrak{a}$ in this theorem is far from being sharp. A more detailed analysis of the functions $\phi_u,\phi_v$ and more refined multilinear analysis (to follow) would certainly yield smaller $\mathfrak{a}$. The importance of this lemma is that it allows one to place the extra derivatives coming from the multiplier to the lower frequency component. This is important when one cannot afford to loose derivative on the high frequency term, such as in energy estimates see \cref{profile energy}.
	\end{remark}
	
	\begin{proof}
		Let $m$ be either of the multipliers. The proof essentially follows that of \cite{tao_harmonic_2009} for the Coifman-Meyer multiplier theorem with a slight modification for how the multiplier, $m$, behaves, since we do not have $\abs{\nabla_\rho^j\nabla_\nu^k m}\lesssim(\abs{\rho}+\abs{\nu})^{-j-k}$. So as not to fully reproduce that proof here, I will mention each step what one needs to do and highlight the part where our case differs. For further detail we recommend to consult with the source.
		
		Let's focus on the first estimate , set $m=\frac{\chi_S}{\phi}$ and parametrise it with input frequencies $\nu_1,\nu_2$, eg. $\phi_u=\jpns{\nu_1+\nu_2}-\jpns{\nu_1}_M+\jpns{\nu_2}$.
		
		\textit{Step 1, splitting}
		We will use Littlewood-Paley theory and split the bilinear estimate into high-high and high-low interactions. Note that the high-high interaction in a sense is much easier, since one is allowed to split the $k$ derivatives between the fields as one wants to. Let $1=\sum_j \psi_j^2$ be a Littlewood-Paley decomposition. And split
		\begin{equation}
			\begin{gathered}
				T_m(f,g)=\pi_{hh}+\pi_{hl}+\pi_{lh}:=\sum_{j,k;j=k+\mathcal{O}(1)}T_m(\psi_j^2(D)f,\psi_k^2(D)g)\\
				+\sum_{j,k;j-k>\mathcal{O}(1)}T_m(\psi_j^2(D)f,\psi_k^2(D)g)+\sum_{j,k;k-j>\mathcal{O}(1)}T_m(\psi_j^2(D)f,\psi_k^2(D)g).
			\end{gathered}
		\end{equation}
		
		\textit{Step 2, high-high} Since, we are restricting the frequency ranges, we have by Minkowski inequlaity
		\begin{equation}
			\begin{gathered}
				\abs{\pi_{hh}(f,g)}\leq\sum_{j,k;j=k+\mathcal{O}(1)}\abs{T_{m_{jk}}(\psi_j(D)f,\psi_k(D)g)}
			\end{gathered}
		\end{equation}
		for $m_{jk}=m\psi_j\psi_k$. Using that $m_{jk}$ has finite support, we can consider it as a periodic function in a box ($V_{jk}$) of side lengths $C2^j$ for $j$ independent $C$ and write 
		\begin{equation}
			\begin{gathered}
				m_{jk}(\nu_1,\nu_2)=\sum_{n_1,n_2\in\Z^3}c_{n_1,n_2}e^{i(n_1\nu_1+n_2\nu_2)/C2^j}
			\end{gathered}
		\end{equation}
		where the Fourier coefficients are given by
		\begin{equation}
			\begin{gathered}
				c_{n_1,n_2}=\frac{1}{\abs{V_{jk}}(2\pi)^3}\int_{V_{jk}}\d\nu_1\d\nu_2 m_{jk}e^{-i(n_1\nu_1+n_2\nu_2)/C2^j}.
			\end{gathered}
		\end{equation}
		
		Now we deviate from the proof of \cite{tao_harmonic_2009}. In that, he uses the fact that $c_{n_1,n_2}\lesssim(1+n_1+n_2)^{-100}$, but this only holds because one may integrate by parts and using the Coifman-Meyer bounds do not pick up $j,k$ dependent constants on the way. Indeed, doing integration by part $s$ times, we get
		\begin{equation}
			\begin{gathered}
				c_{n_1,n_2}=\frac{1}{\abs{V_{jk}}(2\pi)^3}\int_{V_{jk}}\d\nu_1\d\nu_2 \nabla_{\nu_1}^sm_{jk}\Big(\frac{C2^j}{in_1}\Big)^se^{-i(n_1\nu_1+n_2\nu_2)/C2^j}.
			\end{gathered}
		\end{equation}
		We may crudely bound the $\nabla_{\nu_1}^sm_{jk}$ using \cref{symbolic estimates}, concluding $\abs{c_{n_1,n_2}}\lesssim_{s,C}\frac{2^{j(2s+1)}}{n_1^s}$ where we also ignored the fact that support of $\chi_S$ is very small within the area of integration (further improvement possibility on the value of $a$). The same bound also applies if we integrate in the $\nu_2$ direction with $n_1$ replaced by $n_2$ in the integral:
		\begin{equation}\label{c decay hh}
			\begin{gathered}
				\abs{c_{n_1,n_2}}\lesssim_{s,C} 2^{j(2s+1)}(1+\abs{n_1}+\abs{n_2})^{-s},
			\end{gathered}
		\end{equation}
		$\forall s$. Now we return to the proof of \cite{tao_harmonic_2009}. Applying the Fourier decomposition of $m_{jk}$ we get
		\begin{equation}
			\begin{gathered}
				T_{m_{jk}}(\psi_j(D)f,\psi_k(D)g)(x)=\sum_{n_1,n_2}c_{n_1,n_2}\psi_j(D)f(x-n_1/C2^j)\psi_j(D)g(x-n_1/C2^k)
			\end{gathered}
		\end{equation}
		
		At this point, one uses the the almost stationarity of Littlewood-Paley projected functions (Lemma 3.1 in \cite{tao_harmonic_2009})
		\begin{equation}
			\begin{gathered}
				\abs{\psi_j(D)f(x-n_1/C2^j)}\lesssim(1+\abs{n_1})^3M(\psi_j(D)f)(x)
			\end{gathered}
		\end{equation} 
		where $M$ is the Hardy-Littlewood maximal functional. This gives the pointwise estimate
		\begin{equation}
			\begin{gathered}
				\abs{\pi_{hh}(f,g)}\lesssim\sum_{j,k;j=k+\mathcal{O}(1)}\sum_{n_1,n_2\in\Z^3}c_{n_1,n_2}(1+\abs{n_1})^3(1+\abs{n_2})^3\abs{M\psi_j(D)f}\abs{M\psi_k(D)g},
			\end{gathered}
		\end{equation}
		which is going to be summable in $n_1,n_2$ using \cref{c decay hh} given certain loss of derivatives ($s>12$).
		
		Using Cauchy-Schwartz and Holder gives
		\begin{equation}
			\begin{gathered}
				\norm{\pi_{hh}(f,g)}_{W^{k,r}}\lesssim\norm{(\sum_j\abs{M\psi_j(D)f}^2)^{1/2}}_{W^{k,p}}\norm{(\sum_j\abs{M\psi_j(D)g}^2)^{1/2}}_{W^{2s+1,q}}.
			\end{gathered}
		\end{equation}
		The result follow by Fefferman-Stein and Littlewood-Paley inequality.

		\textit{Step 3, low-high}
		
		The proof is similar to the previous case, though one has to be more careful about getting the derivative loss on the right component. After a Littlewood-Paley decomposition and using the high-low nature of the multiplier, we get
		\begin{equation}
			\begin{gathered}
				\pi_{lh}(f,g)=\sum_j\pi_{lh}(f,\psi_k(D)\psi_k(D)g)=\sum_{k,j;j<k-\mathcal{O}(1)}T_{m_jk}(\psi_j(D)f,\psi_k(D)g).
			\end{gathered}
		\end{equation}
		The $n_1$ bound on the coefficients follow as before, because in that case we explicitly get low frequency losses. For the $n_2$ case, we have to bound $\nabla_{\nu_2}^sm_{jk}2^{sk}$, which looks troubling at first, due to the loss of high frequency derivative terms, but note that from \cref{symbolic estimates}, we obtain 
		\begin{equation}
			\begin{gathered}
				\abs{c_{n_1,n_2}}\lesssim2^{j(2s+1)}(1+\abs{n_1}+\abs{n_2})^{-s}.
			\end{gathered}
		\end{equation}
		We may finish the argument in a similar manner as before. Let's put the loss of derivatives on the $f$ factor and perform the sum to get $\psi_{\leq k-C}(1-\Delta)^{s+1}f$ for $C$ being the separation between low-high frequencies. Then we finish the proof the same way as in \cite{tao_harmonic_2009}, via using Holder, Hardy-Littlewood, Ferfferman-Stein and Littlewood-Paley.
		
		We do not need to consider high-low interactions, as $\chi_S$ has no support there.
		
		\textit{Step 4 other cases}	
		The rest of multipliers $\phi_v,\nabla\phi_v,\nabla\phi_u$ follow a similar line of reasoning and the only difference is how one obtains the bounds for $c_{n_1,n_2}$, but all will follow from \cref{symbolic estimates}.
	\end{proof}

	\begin{remark}
		An alternative way to prove the above assertion would be to show that the Fourier transform of $\frac{1}{\phi\min(\abs{\nu_1},\abs{\nu_2})^a}$ is in $L^1$. 
	\end{remark}
	
	This lemma essentially tells us, that the derivatives from the multiplier will not be put on the high frequency term, instead, we can put them on the low frequency part. Of course, the above lemma could be weakened to the case, where all the derivatives are on the highest order term.
	
	Finally, let me note, that choosing different signs (values for $\epsilon$) in the definition of $\phi$, we may get rid of resonances, but the asymptotic degeneration of $\phi$ and $\nabla_\nu\phi$ will remain the same, therefore, one cannot hope to get better estimates than the ones above.

	\section{Proof of \cref{a priori bounds}}\label{a priori bound section}
	Throughout this section, we will be using the estimates from \cref{auxiliary tools} and the bootstrap assumptions \cref{bootstrap assumption} without referencing them everywhere. As all Sobolev embeddings and integral estimates depend on the exponents that we choose for the bootstrap ($N,k,\delta$), we will not keep repeating them for each inequality.
	\subsection{Energy estimates in $\norm{\cdot}_{H^N}$}\label{energy}
	\subsubsection{Light field ($u$)}
	We need to prove $\frac{c_E}{M^1.5}$ bound.
	
	Using Minkowski inequality, product estimate, Sobolev embedding and the bootstrap assumptions yield
	\begin{equation}
		\begin{gathered}
			\norm{L}_{H^N}\lesssim\intt\norm{uv}_{H^N}\lesssim\intt\norm{u}_{H^N}\norm{v}_{L^\infty}+\norm{u}_{L^\infty}\norm{v}_{H^N}\\ \lesssim
			\intt \norm{u}_{H^N}\norm{v}_{W^{1/2,(\frac{1}{6}-\delta)^{-1}}}+\norm{u}_{W^{1/2,(\frac{1}{6}-\delta)^{-1}}}\norm{v}_{H^N}
			\\\lesssim \frac{E^2}{M^{2.5}}\intt\Big(\jpns{\frac{s}{M}}^{-1-3\delta}+\jpns{s}^{-1-3\delta}\Big)\lesssim\frac{E^2}{M^{1.5}}.
		\end{gathered}
	\end{equation}
	\subsubsection{Heavy field ($v$)}
	We need to prove $\frac{c_E}{M^1.5}$ bound.
	
	Using Minkowski inequality, product estimate and bootstrap assumption we obtain
	\begin{equation}
		\begin{gathered}
			\norm{H}_{H^N}\lesssim\intt\norm{uu}_{H^N}\lesssim\intt\norm{u}_{H^N}\norm{u}_{W^{1/2,(\frac{1}{6}-\delta)^{-1}}}\lesssim\frac{E^2}{M^2}\intt\jpns{s}^{-1-3\delta}\lesssim\frac{E^2}{M^2}.
		\end{gathered}
	\end{equation}
	Here, to close the bootstrap, it's really important to have  only $\sob{v}{N}=\mathcal{O}(M^{-1.5})$ instead of $M^{-2}$.
	
	\subsection{Decay estimate in $\norm{e^{it\jpns{D}}\cdot}_{W^{k,(\frac{1}{6}-\delta)^{-1}}}$}\label{decay}
	\subsubsection{Light field}
	We need to prove $\frac{c_E}{M^1.5}\jpns{t}^{-1-3\delta}$ bound.
	
	Let us split the frequencies as explained in \cref{space time method}
	\begin{equation}
		\begin{gathered}
			L=\intl T_1(u,v)=\intl(T_{\chi_S}(u,v)+T_{\chi_T}(u,v))=I+II.\\
		\end{gathered}
	\end{equation}
	First focus on space resonant set, $I$, and do an integration by parts with respect to $s$ using $\frac{1}{i\phi}\partial_s e^{is\phi}=e^{is\phi}$:
	\begin{equation}
		\begin{gathered}
			I=e^{-it\jpns{D}}T_{\frac{\chi_S}{i\phi}}(u,v)|_t-e^{-i\jpns{D}}T_{\frac{\chi_S}{i\phi}}(u,v)|_1-\intl(T_{\frac{\chi_S}{i\phi}}(u,uu)+T_{\frac{\chi_S}{i\phi}}(uv,v)).
		\end{gathered}
	\end{equation}
	Note, that the frequencies that one picks up from $\phi$ and the derivatives of $l,h$ are exactly so that we end up with the above expression, that is amenable for traditional decay estimates. We begin by bounding the boundary term with product estimate \cref{multiplier lemma}, Sobolev embedding:
	
	\begin{equation}
		\begin{gathered}
			\norm{T_{\frac{\chi_S}{i\phi}}(u,v)|_t}_{W^{k,(\frac{1}{6}-\delta)^{-1}}}\lesssim\norm{u}_{W^{k,(\frac{1}{6}-\delta)^{-1}}}\norm{v}_{W^{\mathfrak{a},\infty}}+\norm{v}_{W^{k,(\frac{1}{6}-\delta)^{-1}}}\norm{u}_{W^{\mathfrak{a},\infty}}\\ \lesssim\norm{u}_{W^{k+3/4,3}}\norm{v}_{W^{\mathfrak{a}+1/2,(\frac{1}{6}-\delta)^{-1}}}+\norm{v}_{W^{k,(\frac{1}{6}-\delta)^{-1}}}\norm{u}_{W^{\mathfrak{a}+1/2,(\frac{1}{6}-\delta)^{-1}}}\\
			\lesssim \frac{E}{M\jpns{t}^{1+3\delta}}\frac{E}{M^{1.5}\jpns{t/M}^{1+3\delta}}\lesssim\jpns{t}^{-1-3\delta}\frac{E^2}{M^{2.5}}.
		\end{gathered}
	\end{equation}
	Where we've also used $k>\mathfrak{a}+1/2$.	For the other boundary term, we only need to use the $e^{it\jpns{D}}:W^{k,p}\to t^{-3(1/2-1/p)}L^{p^\prime}$ decay estimate \cref{decay estimate}:
	\begin{equation}
		\begin{gathered}
			\norm{e^{(t-1)i\jpns{D}}T_{\frac{\chi_S}{i\phi}}(u,v)|_1}_{W^{k,(\frac{1}{6}-\delta)^{-1}}}\lesssim\jpns{t-1}^{-1-3\delta}\norm{T_{\frac{\chi_S}{i\phi}}(u,v)}_{W^{k+3/2,(\frac{5}{6}+\delta)^{-1}}}\lesssim\frac{E^2}{M^{2.5}\jpns{t}^{1+3\delta}}
		\end{gathered}
	\end{equation}
	The bulk contribution can be bounded using Minkowski inequality, decay estimate \cref{decay estimate}, product estimate \cref{multiplier lemma} and M-dependent integrals \cref{integral estimate}
	\begin{equation}
		\begin{gathered}
			\norm{\intt e^{i(t-s)\jpns{D}}T_{\frac{\chi_S}{i\phi}}(uv,v)}_{W^{k,(\frac{1}{6}-\delta)^{-1}}}\lesssim\intt \norm{e^{i(t-s)\jpns{D}}T_{\frac{\chi_S}{i\phi}}(uv,v)}_{W^{k,(\frac{1}{6}-\delta)^{-1}}}\\
			\lesssim\intt \jpns{t-s}^{-1-3\delta}\norm{T_{\frac{\chi_S}{i\phi}}(uv,v)}_{W^{k+3/2,(\frac{5}{6}+\delta)^{-1}}} \\ \lesssim \intt \jpns{t-s}^{-1-3\delta}(\norm{u}_{H^{N}}\norm{v}_{W^{\mathfrak{a},(\frac{1}{6}+\delta/2)^-1}}^2+\norm{v}_{H^N}\norm{u}_{W^{\mathfrak{a},(\frac{1}{6}+\delta/2)^-1}}\norm{v}_{W^{\mathfrak{a},(\frac{1}{6}+\delta/2)^-1}})
			\\ \lesssim\intt \jpns{t-s}^{-1-3\delta}(\jpns{s/M}^{-2+3\delta}+\jpns{s}^{-1+3\delta/2}\jpns{s/M}^{-1+3\delta/2})\frac{E^3}{M^{4}}\lesssim \jpns{t}^{-1-3\delta}\frac{E^3}{M^{3-3\delta}}.
		\end{gathered}
	\end{equation}
	Where we've used that $-2+3\delta<-1-3\delta$. The other contribution is dealt with in a very similar manner:
	\begin{equation}
		\begin{gathered}
			\norm{\intt e^{i(t-s)\jpns{D}}T_{\frac{\chi_S}{i\phi}}(uu,u)}_{W^{k,(\frac{1}{6}-\delta)^{-1}}}\lesssim\intt \norm{e^{i(t-s)\jpns{D}}T_{\frac{\chi_S}{i\phi}}(uu,u)}_{W^{k,(\frac{1}{6}-\delta)^{-1}}}\\
			\lesssim\intt \jpns{t-s}^{-1-3\delta}\norm{T_{\frac{\chi_S}{i\phi}}(uu,u)}_{W^{k+3/2,(\frac{5}{6}+\delta)^{-1}}} \\ \lesssim \intt \jpns{t-s}^{-1-3\delta}\norm{u}_{H^{N}}\norm{u}_{W^{\mathfrak{a},(\frac{1}{6}+\delta/2)^-1}}^2
			\\ \lesssim\intt \jpns{t-s}^{-1-3\delta}\jpns{s}^{-2+3\delta}\frac{E^3}{M^{3}}\lesssim \jpns{t}^{-1-3\delta}\frac{E^3}{M^{3}}.
		\end{gathered}
	\end{equation}
	Now we turn our attention to the time resonant contribution. Performing an integration by parts with respect to $\nu$, using $\frac{\nabla_\nu\phi}{\abs{\nabla_\nu\phi}^2}\cdot\nabla_\nu e^{is\phi}=e^{is\phi}$
	\begin{equation}
		\begin{gathered}
			\widehat{II}=\intt\int\d\nu \frac{e^{-is\phi}}{s} \nabla_\nu\Big(m\hat{h}(\nu)\hat{l}(\rho-\nu)\Big)=
			\intt\int\d\nu \frac{e^{-is\phi}}{s} \Big(m\hat{h}\nabla_\nu\hat{l}+m\hat{l}\nabla_\nu\hat{h}+\hat{h}\hat{l}\nabla_\nu m\Big)
		\end{gathered}
	\end{equation}
	for $m=\frac{\chi_T\nabla_\nu\phi}{\abs{\nabla_\nu\phi}^2}$. The part, where $\nabla$ acts on the multiplier is easier as we have better decay for the profiles in that case, therefore we will only show bounds for the other two terms. This will follow from Minkowski inequality, decay estimate, product estimate, interpolation inequality
	\begin{equation}
		\begin{gathered}
			\norm{\intt\frac{e^{i(t-s)\jpns{D}}}{s}T_{m}(v,e^{is\jpns{D}}xl)}_{W^{k,(\frac{1}{6}-\delta)^{-1}}}\lesssim\intt\frac{ 1}{\jpns{t-s}^{1+3\delta}s}\norm{T_m(v,e^{is\jpns{D}}xl)}_{W^{k+3/2,(\frac{5}{6}+\delta)^{-1}}}\\
			\lesssim\intt\frac{ 1}{\jpns{t-s}^{1+3\delta}s}(\norm{v}_{W^{k+3/2,(\frac{1}{3}+\delta)^{-1}}}\norm{e^{is\jpns{D}}xl}_{H^{\mathfrak{a}}}+\norm{v}_{W^{\mathfrak{a},(\frac{1}{3}+\delta)^{-1}}}\norm{e^{is\jpns{D}}xl}_{H^{k+3/2}})\\
			\lesssim\intt\frac{ 1}{\jpns{t-s}^{1+3\delta}s}\norm{v}^{17/20}_{W^{k,(\frac{1}{3}+\delta)^{-1}}}\norm{v}^{3/20}_{W^{k+10,(\frac{1}{3}+\delta)^{-1}}}\norm{e^{is\jpns{D}}xl}_{H^{k+3/2}}	\\
			\lesssim\intt\frac{1}{\jpns{t-s}^{1+3\delta}s}\frac{1}{\jpns{s/M}^{(1/2-3\delta)17/20}}\frac{E^2}{M^{2.5}}\lesssim\jpns{t}^{-1-3\delta}\frac{E^2}{M^{2.5-3\delta}}.
		\end{gathered}
	\end{equation}
	Where, we needed that $1+3\delta<1+17/20(1/2-3\delta)$ and $N>k+3/2+10$. The other term is dealt with similarly:
	\begin{equation}
		\begin{gathered}
			\norm{\intt\frac{e^{i(t-s)\jpns{D}}}{s}T_{m}(u,e^{is\jpns{D}_M}xh)}_{W^{k,(\frac{1}{6}-\delta)^{-1}}}\lesssim\intt\frac{ 1}{\jpns{t-s}^{1+3\delta}s}\norm{T_m(u,e^{is\jpns{D}_M}xh)}_{W^{k+3/2,(\frac{5}{6}+\delta)^{-1}}}\\
			\lesssim\intt\frac{ 1}{\jpns{t-s}^{1+3\delta}s}\norm{u}_{W^{k+3/2,(\frac{1}{3}+\delta)^{-1}}}\norm{e^{is\jpns{D}_M}xh}_{H^{k+3/2}}\\
			\lesssim\intt\frac{1}{\jpns{t-s}^{1+3\delta}s}\frac{1}{\jpns{s}^{(1/2-3\delta)17/20}}\frac{E^2}{M^{2.5}}\lesssim\jpns{t}^{-1-3\delta}\frac{E^2}{M^{2.5}}
		\end{gathered}
	\end{equation}

	\subsubsection{Heavy field}
	We need to prove $\frac{c_E}{M^1.5}\jpns{t/M}^{-1-3\delta}$ bound.
	
	We follow the same steps as for the light field.
	\begin{equation}
		\begin{gathered}
			H=\int_{1}^{t}\d s  e^{-is\jpns{D}_M}T_1(u,u)=\int_{1}^{t}\d s  e^{-is\jpns{D}_M}(T_{\chi_S}(u,u)+T_{\chi_T}(u,u))=:I+II
		\end{gathered}
	\end{equation}
	Let's start with the space resonance
	\begin{equation}
		\begin{gathered}
			\inth T_{\chi_S}(u,u)=e^{-it\jpns{D}_M}T_{\frac{\chi_S}{i\phi}}(u,u)|_t-e^{-i\jpns{D}_M}T_{\frac{\chi_S}{i\phi}}(u,u)|_1-2\inth T_{\frac{\chi_S}{i\phi}}(u,uv)
		\end{gathered}
	\end{equation}
	Boundary:
	
	\begin{equation}
		\begin{gathered}
			\norm{T_{\frac{\chi_S}{i\phi}}(u,u)|_t}_{W^{k,(\frac{1}{6}-\delta)^{-1}}}\lesssim\norm{u}_{W^{k,\infty}}\norm{u}_{W^{\mathfrak{a},(\frac{1}{6}-\delta)^{-1}}}\\\lesssim\norm{u}^{17/20}_{W^{k,(\frac{1}{6}-\delta)^{-1}}}\norm{u}^{3/20}_{H^{N}}\norm{u}_{W^{\mathfrak{a},(\frac{1}{6}-\delta)^{-1}}}
			\lesssim\jpns{t}^{-1-3\delta}\frac{E^2}{M^{2}}
		\end{gathered}
	\end{equation}
	$N>k+11$
	Bulk:
	\begin{equation}
		\begin{gathered}
			\norm{\intt e^{i(t-s)\jpns{D}_M}T_{\frac{\chi_S}{i\phi}}(u,uv)}_{W^{k,(\frac{1}{6}-\delta)^{-1}}}\lesssim\intt \jpns{\frac{t-s}{M}}^{-1-3\delta}\norm{T_{\frac{\chi}{i\phi}}(u,uv)}_{W^{k,(\frac{5}{6}+\delta)^{-1}}}\\\lesssim
			\intt\jpns{\frac{t-s}{M}}^{-1-3\delta} (\norm{u}_{H^N}\norm{u}_{W^{\mathfrak{a},6}}\norm{v}_{W^{\mathfrak{a},(\frac{1}{6}+\delta)^{-1}}}+\norm{v}_{H^N}\norm{u}^2_{L^(\frac{1}{6}+\delta/2)^{-1}})
			\\\lesssim\intt\jpns{\frac{t-s}{M}}^{-1+3\delta}\frac{E^3}{M^{3.5}}(\jpns{s}^{-1+3\delta}\jpns{s/M}^{-1+3\delta}+\jpns{s}^{-2+3\delta})\lesssim\jpns{t/M}^{-1-3\delta}  \frac{E^3}{M^{2.5}}
		\end{gathered}
	\end{equation}
	Now we turn our attention to $II$.
	\begin{equation}
		\begin{gathered}
			\widehat{II}=-\intt \int\d\nu \frac{e^{is\phi}}{s}\nabla_\nu\Big(m\hat{l}(\nu)\hat{l}(\rho-\nu)\Big)\\
			=-\intt \int\d\nu \frac{e^{is\phi}}{s}(\hat{l}(\nu)\hat{l}(\rho-\nu)\nabla_\nu m+m\hat{l}(\rho-\nu)\nabla_\nu\hat{l}(\nu)-m\hat{l}(\nu)\nabla_\nu\hat{l}(\rho-\nu))
		\end{gathered}
	\end{equation}
	where $m=\frac{\chi_T\nabla_\nu\phi}{\abs{\nabla_\nu\phi}^2}$. As before, we will only work with the case where $\nabla_\nu$ act on a phase rather than the multiplier.
	
	\begin{equation}
		\begin{gathered}
			\norm{\intt  \frac{e^{i(t-s)\jpns{D}_M}}{s}T_m(u,e^{is\jpns{D}}xl)}_{W^{k,(\frac{1}{6}-\delta)^{-1}}}\lesssim\int_{0}^{t}\d s\frac{1}{s\jpns{(t-s)/M}^{1+3\delta}}\norm{T_m(u,e^{is\jpns{D}}xl)}_{W^{k+3/2,(\frac{5}{6}+\delta)^{-1}}}
			\\\lesssim\int_{0}^{t}\d s\frac{1}{s\jpns{(t-s)/M}^{+1+3\delta}} \norm{u}_{W^{k+3/2,(1/3+\delta)^{-1}}}\norm{e^{is\jpns{D}}xl}_{H^{\mathfrak{a}}}
			\\ \lesssim\intt \frac{1}{s\jpns{(t-s)/M}^{+1+3\delta}}\frac{1}{\jpns{s}^{(1/2-3\delta)17/20}}\frac{E^2}{M^2}\lesssim\frac{E^2}{M^2}\jpns{t/M}^{-1-3\delta}
		\end{gathered}
	\end{equation}

	\subsection{Energy estimate $\norm{x\cdot}_{H^{k+3/2}}$}\label{profile energy}
	We need to prove $\frac{c_E}{M^1.5}$ bound.
	
	This is the only section, where we will use crucially that the derivatives are lost on the low frequency term when using the multipliers coming from the integration by parts argument.  
	\subsubsection{Light field}
	We know that an $x$ weight in physical space corresponds to $\nabla_\rho$ in Fourier, so we need to bound
	\begin{equation}\label{split a1 a2 a3}
		\begin{gathered}
			\nabla_\rho L=\intt e^{-is\jpns{D}}\Big(sT_{-i(\nabla_\rho\phi)}(u,v)+T(e^{is\jpns{D}}xl,v)+T(u,e^{is\jpns{D}_M}xh)\Big)=a_1+a_2+a_3
		\end{gathered}
	\end{equation}
	Since $\nabla_\rho\phi$ has all the desired decay properties and we don't want to use any gain from such a factor, we will drop it in the following discussion.
	Consider the splitting to resonances $a_i^I,a_i^{II}$, as done above. Focusing on $a_1$, we can bound each term as follows:
	\begin{itemize}
		\item Boundary term coming from integration by parts with respect to $s$
		\begin{equation}
			\begin{gathered}
				\norm{tT_{\frac{\chi_S}{i\phi}}(u,v)}_{H^{k+3/2}}\lesssim t\norm{u}_{W^{\mathfrak{a},6}}\norm{v}_{W^{k+3,2}}+t\norm{u}_{W^{N,2}}\norm{v}_{W^{\mathfrak{a},6}}
				\\\lesssim t\frac{E^2}{M^{2.5}}(\jpns{t}^{-1}+\jpns{t/M}^{-1})\lesssim\frac{E^2}{M^{1.5}}
			\end{gathered}
		\end{equation}
	\item Bulk term from of $a_I$ after integration by parts when $\partial_s$ acts on the light field	
	\begin{equation}
		\begin{gathered}
			\norm{\int_{1}^{t}\d s sT_{\frac{\chi_S}{i\phi}}(uv,v)}_{W^{k+3/2,2}}\lesssim\intt  s\norm{T_{\frac{\chi_S}{i\phi}}(uv,v)}_{W^{k+3/2,2}}\\
			\lesssim\intt_1^ts(\norm{u}_{W^{k+3/2,2}}\norm{v}_{W^{\mathfrak{a}+1/2,(1/6-\delta)^{-1}}}^2+\norm{v}_{W^{k+3/2,2}}\norm{v}_{W^{\mathfrak{a}+1/2,(1/6-\delta)^{-1}}}\norm{u}_{W^{\mathfrak{a}+1/2,(1/6-\delta)^{-1}}})
			\\\lesssim\intt \frac{s}{\jpns{s/M}^{2+6\delta}}\frac{E^3}{M^4}\lesssim\frac{E^3}{M^{2}}
		\end{gathered}
	\end{equation}
		and when $\partial_s$ acts on the heavy field.
	\begin{equation}
			\begin{gathered}
				\norm{\int_{1}^{t}\d sT_{\frac{\chi_S}{i\phi}}(u,uu)}_{W^{k+3/2,2}}s\lesssim\intt s\norm{T_{\frac{\chi_S}{i\phi}}(u,uu)}_{W^{k+3/2,2}}\\ \lesssim
				\int_1^ts\norm{u}_{W^{k+3/2,2}}\norm{u}_{W^{\mathfrak{a}+1/2,(1/6-\delta)^{-1}}}^2\lesssim\intt \frac{1}{\jpns{s}^{1+6\delta}}\frac{E^3}{M^3}\lesssim\frac{E^3}{M^{3}}
			\end{gathered}
	\end{equation}
		\item Time resonant contribution $a_1^{II}$ after integration by parts when $\nabla_\nu$ acts on light field
			\begin{equation}
			\begin{gathered}
				\norm{\intt  T_m(e^{is\jpns{D}}xl,v)}_{W^{k+3/2,2}}\lesssim\int_1^t\norm{T_m(e^{is\jpns{D}}xl,v)}_{W^{k+3/2,2}}\\
				\lesssim\intt \norm{e^{is\jpns{D}}xl}_{H^{k+3/2}}\norm{v}_{W^{\mathfrak{a}+1/2,(1/6-\delta)^{-1}}}+\norm{e^{is\jpns{D}}xl}_{W^{\mathfrak{a},(2/6+\delta)^{-1}}}\norm{v}_{W^{k+3/2,(1/6-\delta)^{-1}}}\\
				\lesssim\intt \norm{e^{is\jpns{D}}xl}_{H^{k+3/2}}\norm{v}_{W^{\mathfrak{a}+1/2,(1/6-\delta)^{-1}}}+\norm{e^{is\jpns{D}}xl}_{W^{\mathfrak{a}+1/2,2}}\norm{v}_{W^{k,(1/6-\delta)^{-1}}}^{17/20}\norm{v}_{W^{k+11,2}}^{3/20}
				\\\lesssim\intt(\frac{1}{\jpns{s/M}^{1+3\delta}}+\frac{1}{\jpns{s/M}^{(1+3\delta)17/20}})\frac{E^2}{M^{2.5}}\lesssim\frac{E^2}{M^{1.5}}
			\end{gathered}
		\end{equation}
		where I've used $N>k+11,\delta>1/17$. When $\nabla_\nu$ acts on heavy field,
		\begin{equation}
			\begin{gathered}
				\norm{\intt  T_m(e^{is\jpns{D}}xh,u)}_{W^{k+3/2,2}}\lesssim\int_1^t\norm{T_m(e^{is\jpns{D}}xh,u)}_{W^{k+3/2,2}}\\
				\lesssim\intt \norm{e^{is\jpns{D}}xh}_{H^{k+3/2}}\norm{u}_{W^{\mathfrak{a}+1/2,(1/6-\delta)^{-1}}}+\norm{e^{is\jpns{D}}xh}_{W^{\mathfrak{a}+1/2,2}}\norm{u}_{W^{k,(1/6-\delta)^{-1}}}^{17/20}\norm{u}_{W^{k+11,2}}^{3/20}
				\\\lesssim\intt(\frac{1}{\jpns{s}^{1+3\delta}}+\frac{1}{\jpns{s}^{(1+3\delta)17/20}})\frac{E^2}{M^{2.5}}\lesssim\frac{E^2}{M^{2.5}}
			\end{gathered}
		\end{equation}
	\end{itemize}
	The remaining terms $a_2,a_3$ are bounded exactly like the last two terms above.

	\subsubsection{Heavy field}
	We need to prove $\frac{c_E}{M^1.5}\jpns{t}^{-1-3\delta}$ bound.
	
	As usual by now, the steps are identical to the previous case, only the factors of $M$ appear at different places. Split $	\nabla_\rho L==a_1+a_2+a_3$ as in \cref{split a1 a2 a3}. Focusing on $a_1$, we have
	\begin{itemize}
		\item Boundary term near space resonance
			\begin{equation}
				\begin{gathered}
					\norm{tT_{\frac{\chi_S}{i\phi}}(u,u)}_{H^{k+3/2}}\lesssim t\norm{u}_{W^{\mathfrak{a}+1/2,6}}\norm{u}_{W^{N,2}}\lesssim t\frac{E^2}{M^{2}}\jpns{t}^{-1}\lesssim\frac{E^2}{M^{2}}
				\end{gathered}
			\end{equation}
		\item Bulk term near space resonance
			\begin{equation}
				\begin{gathered}
					\norm{\intt_{1}^{t} sT_{\frac{\chi_S}{i\phi}}(uv,u)}_{W^{k+3/2,2}}s\lesssim\intt  s\norm{T_{\frac{\chi_S}{i\phi}}(uv,u)}_{W^{k+3/2,2}}\\
					\lesssim\intt_1^ts(\norm{v}_{W^{k+3/2,2}}\norm{u}_{W^{\mathfrak{a}+1/2,(1/6-\delta)^{-1}}}^2+\norm{u}_{W^{k+3/2,2}}\norm{v}_{W^{\mathfrak{a}+1/2,(1/6-\delta)^{-1}}}\norm{u}_{W^{\mathfrak{a}+1/2,(1/6-\delta)^{-1}}})
					\\\lesssim\intt \frac{1}{\jpns{s/M}^{1+3\delta}\jpns{s}^{3\delta}}\frac{E^3}{M^{3.5}}\lesssim\frac{E^3}{M^{2.5}}
				\end{gathered}
			\end{equation}
		\item Bulk term near time resonance
			\begin{equation}
				\begin{gathered}
					\norm{\intt  T_m(e^{is\jpns{D}}xl,u)}_{W^{k+3/2,2}}\lesssim\int_1^t\norm{T_m(e^{is\jpns{D}}xl,u)}_{W^{k+3/2,2}}\\
					\lesssim\intt \norm{e^{is\jpns{D}}xl}_{H^{k+3/2}}\norm{u}_{W^{\mathfrak{a}+1/2,(1/6-\delta)^{-1}}}+\norm{e^{is\jpns{D}}xl}_{W^{\mathfrak{a},(2/6+\delta)}}\norm{u}_{W^{k+3/2,(1/2-\delta)^{-1}}}\\
					\lesssim\intt \norm{e^{is\jpns{D}}xl}_{H^{k+3/2}}\norm{u}_{W^{\mathfrak{a}+1/2,(1/6-\delta)^{-1}}}+\norm{e^{is\jpns{D}}xl}_{H^{\mathfrak{a}+1/2}}\norm{u}_{W^{k,(1/6-\delta)^{-1}}}^{17/20}\norm{u}_{W^{k+11,2}}^{3/20}
					\\\lesssim\intt(\frac{1}{\jpns{s}^{(1+3\delta)17/20}}\frac{E^2}{M^{2}}\lesssim\frac{E^2}{M^{2}}.
				\end{gathered}
			\end{equation}
	\end{itemize}
	The terms $a_2,a_3$ are identical to the time resonance terms.
	
	\section{M expansion and error term}\label{scattering section}
	We will use the initial data norm associated to the 2nd order equation, 
	\begin{equation}
		\begin{gathered}
			\norm{V_0,V_1}_{M,N,k}=\sob{V_0}{N}+\sob{V_1}{N-1}+M\sob{V_0}{N-1}+\sob{xV_0}{k+3/2}+\norm{xV_1}_{H^{k+1/2}}+M\norm{xV_0}_{H^{k+1/2}}.
		\end{gathered}
	\end{equation}
	Note that this norm is the equivalent for the 2nd order system of \cref{weighted initial data norm}.
	
	In this section, we will consider what final state asymptomatic can be derived for the light field if we restrict the initial data of the heavy field further. Let us recall \cref{eom} 
	\begin{equation*}
		\begin{gathered}
			(\Box-1) U= U{V}- U^3/2\\
			(\Box-M^2){V}= U^2{V}- U^4/2+\partial  U\partial  U,
		\end{gathered}
	\end{equation*}
	where $M{U},M^2{V}=\mathcal{O}(1)$ (in a suitable function space). Hence, the dominant forcing term for the ${V}$ equation both in terms of $M$ and decay wise is $\partial{U}\partial{U}$. Indeed, from now on, we will omit the $U^2V$ term as it makes the notation heavier, but adds no more complications, see \cref{dropping v terms}.  However, for ${U}$ equation, both quantities are of the same order in $M$.
	
	Notice that \cref{global existence} may be improved once global existence is known. The forcing for the heavy field has scaling $1/M^2$, so given that the initial data for $V$ has scaling $M^{-\alpha}$ for $\alpha\in[1.5,2]$,  we know that the solution will also scale so.
	\begin{lemma}[Extension of \cref{global existence}]
		Let $N,k$ be as in \cref{global existence}. Given initial data $U_0,U_1,V_0,V_1$ with $\norm{U_0,U_1}_{1,N+1,k+1},M^{\alpha}\norm{V_0,V_1}_{M,N+1,k+1}<E/M$ for $\alpha\in[0.5,1]$ and $M$ sufficiently large, \cref{eom} has global solution with
		$M^{\alpha}(\norm{V,\partial_tV}_{M,N,k})\lesssim_E1$.
	\end{lemma}
	Importantly, note that this means $\sob{V}{N-1}\lesssim1/M^3$. Therefore in one less regular function space, the $UV$ term is one less order in $M$ than the $U^3$, so we can neglect the contribution of $UV$ for high values of $M$.
	
	We now make the change of variables $\bar{V}=V-\sum_{i=1}^n\frac{F_i[U]}{M^{2i}}$ to get
	\begin{equation}\label{eft eom}
		\begin{gathered}
			(\Box-1) U= U\bar{V}-U\sum_{i=1}^n\frac{F_i[U]}{M^{2i}}- U^3/2\\
			(\Box-M^2)\bar{V}=(\Box-M^2)\sum_{i=1}^n\frac{F_i[U]}{M^{2i}}+\partial  U\partial  U-U^4/2=\frac{\Box F_n}{M^{2n}}\\
			{U}(0)=U_0, \partial_t {U}(0)=U_1, \bar{V}(0)=\bar{V}_0,\partial_t\bar{V}(0)=\bar{V}_1\\
		\end{gathered}
	\end{equation}
	where $F_1[U]=\partial U\cdot\partial U-U^4/2$ and $F_{i+1}=\Box F_i$, which are (at least) quadratic. The upshot of course is that we have very good estimates for $U$ and thus if we make a stronger restriction on the initial data for $\bar{V}$ in terms of powers of $M$, that will be propagated. The problem with this approach is that in \cref{eft eom}, we find $\partial_t^nU$ forcing terms for $n>1$. This is an issue, since \cref{eom} cannot give good bound on these terms. The resolution is to introduce EFT conditions on initial data progressively (instead only at a high order) to get more regular time derivatives for $U$. 
	\begin{definition}\label{eft condition}
		We say that an initial data to \cref{eom} satisfies EFT conditions to order $n$ if
				\begin{equation}\label{eft initial data}
			\begin{gathered}
				\norm{U_s,U_{s+1}}_{1,N,k}+\norm{V_s,V_{s+1}}_{M,N,k}<E/M
			\end{gathered}
		\end{equation}
	holds $\forall s\leq n$.
	\end{definition}

	Denoting the right hand side of \cref{eom} with $F_U,F_V$ respectively, we get an analogue of Proposition 7, \cite{reall_effective_2022}:
	\begin{lemma}\label{global existence for higher derivatives}
		Fix $n\geq0$ integer, $N,k,\delta$ as in \cref{global existence} and initial data $U_0,U_1,V_0,V_1$ satisfying EFT conditions to order $2n$. Then
		\begin{equation}\label{eft eom reduced}
			\begin{gathered}
				(\Box-1)\partial_t^sU=\partial_t^s F_U\\
				(\Box-M^2)\partial_t^sV=\partial_t^s F_V\\
				\partial_t^s {U}(0)=U_s, \partial_t^s{V}(0)={V}_s
			\end{gathered}
		\end{equation}
		$\forall s\leq n$ has global solution with bounded $X$ norms after the suitable transformation as detailed in \cref{space time method}, in particular
		\begin{equation}\label{key}
			\begin{gathered}
				\norm{\partial_t^sU}_{H^{N+1}},\jpns{t}^{1+3\delta}\norm{\partial_t^sU}_{W^{k,(1/6-\delta)^{-1}}}\lesssim_E1/M \qquad \forall s\leq n\\
				M\norm{\partial_t^sV}_{H^{N+1}},M^2\norm{\partial_t^sV}_{H^N},M^2\jpns{t}^{1+3\delta}\norm{\partial_t^sV}_{W^{k,(1/6-\delta)^{-1}}}\lesssim_E1/M \qquad \forall s\leq n
			\end{gathered}
		\end{equation}
		and some other top estimates for $M$ sufficiently large.
	\end{lemma}
	
	\begin{proof}
		The $s=0$ case is \cref{global existence}. For higher values of $s$, we proceed by induction. Note that the forcing term contains time derivatives up to order $s+1$, so using the bounds in $\norm{\cdot}_X$ we can treat the time derivatives of order strictly smaller than $s$ as external force with good decay properties, and make the same bootstrap assumption and proceed as in the $s=0$ case.
	\end{proof}
	
	The assumption on initial data in \cref{global existence for higher derivatives} seems really strange, because $U_s,V_s$ for $s>1$ follow from $U_0,U_1,V_0,V_1$ and the equation of motion. 
	To see that the assumption set is not trivial, we will follow the discussion of Lemma 8 in \cite{reall_effective_2022}. 
	
	\subsection{Right initial data}
	Fix $E=1$, ie. make it implicit in inequalities. Let's start with
	\begin{equation*}
		\begin{gathered}
			M\sob{U_0}{N+1},M\sob{U_1}{N},M^2\sob{V_0}{N+1},M^2\sob{V_1}{N}\lesssim1
		\end{gathered}
	\end{equation*}
	focusing only on the non-weighted estimate, the other following analogous argument. Using the equation of motion for $U$, we see that in fact $\norm{U_2}_{H^{N-1}}\lesssim1/M$. To place additional assumption on initial data, let's define $V^0:=V$, $V^1=V^0-F_1[U]/M^2$ as in \cref{eft eom}. We define it's initial data as before using $V^1_n=\partial_t^nV^1(0)$. Since $F_1[U]$ depends at most on first derivative of $U$, at $t=0$ it is a polynomial function of $U_0,U_1$ and their spatial derivatives of $U_0$, denoted by $P_1[U_0,U_1]$. Following \cite{reall_effective_2022}, let's assume\footnote{we find it convenient to introduce a loss of derivative here, as $V^1$ includes derivative terms in its definition}
	\begin{equation}
		\begin{gathered}
			\norm{V^1_0}_{H^N}\lesssim \frac{1}{M^6},
		\end{gathered}
	\end{equation}
	which can be interpreted as a condition on $V_0=(V^0_0)$ by expanding the definition of $V^1_0$. The reason for the factor 6 is that we expect $V_0$ to be given as an expansion in factors of $M^2$ in $U_0,U_1$, which themselves are suppressed in our discussion by a factor of $M$. Using the bound on $U_0,U_1$, and triangle inequality we get 
	\begin{equation}
		\begin{gathered}
			M^4\sob{V^0_0}{N}\leq	M^4\norm{V^1_0}_{H^N}+M^2\sob{P_1[U_0,U_1]}{N}\lesssim1.
		\end{gathered}
	\end{equation}
	Using the equation of motion for $V^0$, $\partial_t^2V^0=\Delta V^0+M^2(F_1[U]/M^2-V^0)$,  we also get that $\sob{V_2}{N-2}M^4\lesssim1$. Using equation of motion for $U$ again, we get $\sob{U_3}{N-2}M\lesssim1$. We can also iterate this process
	\begin{lemma}
		Fix $n\geq0$ integer and $N-n-2>2$. If $M\sob{U_j}{N+1-j},M^4\sob{V^0_j}{N-j}\leq1$ $\forall j\leq n+1$ and $\sob{V^1_n}{N-n-2}\leq1/M^6$, then  $M\sob{U_{n+2}}{N-n-1},M\sob{U_{n+3}}{N-n-2},M^4\sob{V^0_{n+2}}{N-n-2}\lesssim1$.
	\end{lemma}
	\begin{proof}
		The restriction in the statement are to require that the Sobolev spaces used form an algebra.
		
		$M\sob{U_{n+2}}{N-n-1}$ bound follows from the assumptions on $U_j,V^0_j$ and the equation of motion
		\begin{equation*}
			\begin{gathered}
				\partial_t^2\partial_t^{n}U=\Delta \partial_t^{n}U+\partial_t^{n}F_U.
			\end{gathered}
		\end{equation*}
		The second bound for $U$ will follow similarly, once the $V^0$ bound is established.
		For $V^0$, note that the equation of motion is
		\begin{equation*}
			\begin{gathered}
				\partial_t^2\partial_t^n V^0=\Delta\partial_t^n V^0+M^2\partial_t^n(F_1[U]/M^2-V^0).
			\end{gathered}
		\end{equation*}
		Taking $\sob{\cdot}{N-n-2}$ norm establishes the result.
	\end{proof}
	
	This work is a partial resolution, as we impose constraints on lower time derivative fields. To extend this process, define $V^n=V^{n-1}-F_n[U]/M^{2n}$ iteratively ($F_n$ as in \cref{eft eom}), where $F_n[U]$ contains at most $2n-1$ time derivative of $U$ terms, which in turn may be expanded in terms of $U_0,U_1,V^0_0,V^0_1$. Importantly, if we restrict $\norm{V^2_0,V^2_1}_{H^N\times H^{N-1}}M^8\lesssim1$, then that is amenable to an expansion in $1/M$ for $U^0_0,U^0_1$. Which will yield a statement
	\begin{equation*}
		\begin{gathered}
			M^8\sob{V^0_0-\frac{P_1[U_0,U_1]}{M^2}-\frac{P_2[U_0,U_1]}{M^4}}{N}\lesssim1\\
			M^8\sob{V^0_1-\frac{\tilde{P}_1[U_0,U_1]}{M^2}-\frac{\tilde{P}_2[U_0,U_1]}{M^4}}{N-1}\lesssim1
		\end{gathered}
	\end{equation*}
	for some functional $P_2$ that involve at most 3 spatial derivatives of $U_0$ and two of $U_1$.	Using the $V^i$ we can reduce the number of time derivatives in a a condition:
	
	\begin{lemma}
		Fix $n,m-1\geq0$ integers, $N-n-2m>2$. If 
		\begin{equation*}
			\begin{gathered}
				M^{4+2m}\sob{\partial_t^n V^m}{N-n-2m},M^{2+2m}\sob{\partial_t^{n} V^{m-1}}{N+2-n-2m}\leq1\\
				M\sob{U_j}{N+1-j}\leq1 \qquad\forall j\leq2m-1+n,
			\end{gathered}
		\end{equation*}
		 then $M^{2+2m}\sob{\partial_t^{n+2} V^{m-1}}{N-n-2m}\lesssim1$.
	\end{lemma}
	\begin{proof}
		The equation of motion for $V^{m-1}$ is 
		\begin{equation*}
			\begin{gathered}
				\partial_t^2V^{m-1}=\Delta V^{m-1}+(F_m[U]/M^2-V^{m-1})M^2,
			\end{gathered}
		\end{equation*}
		where $F_m$ contains at most $2m-1$ time derivatives on $U$. Differentiating this equation $n$ times and using triangle inequality yields the required result.
	\end{proof}
	
	Using these two iterative processes, we conclude
	\begin{lemma}\label{eft initial data lemma}
		Fix $n\geq0$, $N-n-2m>2$. For $M\sob{U_0}{N+2n+1},M\sob{U_1}{N+2n}\leq 1$, there exist functionals $P_i,\tilde{P}_i$ containing at most $2i-1$ ($2i-2$) spatial derivatives of $U_0$ ($U_1$), such that
		\begin{equation}\label{exact eft initial data}
			\begin{gathered}
				\sob{V_0-\frac{P_1[U_0,U_1]}{M^2}-...-\frac{P_j[U_0,U_1]}{M^{2j}}}{N+2n-2j}M^{4+2j}\lesssim 1\\
				\sob{V_1-\frac{\tilde{P}_1[U_0,U_1]}{M^2}-...-\frac{\tilde{P}_j[U_0,U_1]}{M^{2j}}}{N+2n-2j-1}M^{4+2j}\lesssim 1
			\end{gathered}
		\end{equation}
		$\forall j\leq n$ imply $M^4\sob{V_j}{N+2n-j},M\sob{U_j}{N+2n+1-j}\lesssim1$ for all $j\leq2n$.
	\end{lemma}
	\begin{proof}
		This is an immediate consequence of the previous results. One needs to expand $\sum_{i=1}^n\frac{F_i[U]}{M^{2i}}$ at $t=0$ in terms of the initial data $U_0,U_1,V_0,V_1$ and then solve for $V_0,V_1$ in powers of $M$ with an error term of order $M^{2n}$. See \cref{iteration process for eft data}
	\end{proof}
	
	\begin{table}[h!]
		\centering
		\begin{tabular}{cccc}
			$\sob{\partial_t^4U_0}{N+1}$\tikzmark{u4} & $\sob{\partial_t^4V^0_0}{N}$\tikzmark{v04r} & $\sob{\partial_t^4V^1_0}{N-2}$ & $\sob{\partial_t^4V^2_0}{N-4}$\\
			$\sob{\partial_t^3U_0}{N+2}$\tikzmark{u3} & \tikzmark{v03l}$\sob{\partial_t^3V^0_0}{N+1}$\tikzmark{v03r} & $\sob{\partial_t^3V^1_0}{N-1}$\tikzmark{v13r} & $\sob{\partial_t^3V^2_0}{N-3}$\\
			$\sob{\partial_t^2U_0}{N+3}$\tikzmark{u2} & \tikzmark{v02l}$\sob{\partial_t^2V^0_0}{N+2}$\tikzmark{v02r} & \tikzmark{v12l}$\sob{\partial_t^2V^1_0}{N}$\tikzmark{v12r} & $\sob{\partial_t^2V^2_0}{N-2}$\\
			$\sob{\partial_t^1U_0}{N+4}$\tikzmark{u1} & \tikzmark{v01l}$\sob{\partial_t^1V^0_0}{N+3}$\tikzmark{v01r} & \tikzmark{v11l}$\sob{\partial_t^1V^1_0}{N+1}$ & \tikzmark{v21l}$\sob{\partial_t^1V^2_0}{N-1}$ \\
			$\sob{U_0}{N+5}$ & $\sob{V^0_0}{N+4}$\tikzmark{v00r} & \tikzmark{v10l}$\sob{V^1_0}{N+2}$ & \tikzmark{v20l}$\sob{V^2_0}{N}$ \\
		\end{tabular}
		\begin{tikzpicture}[overlay, remember picture, shorten >=.5pt, shorten <=.5pt, transform canvas={yshift=.25\baselineskip}]
			\draw [->] ({pic cs:u1}) [bend right] to ({pic cs:u2});
			\draw [->] ({pic cs:v01l}) to ({pic cs:u2});
			\draw [->] ({pic cs:u2}) [bend right] to ({pic cs:u3});
			\draw [->] ({pic cs:v02l}) to ({pic cs:u3});
			\draw [->] ({pic cs:u3}) [bend right] to ({pic cs:u4});
			\draw [->] ({pic cs:v03l}) to ({pic cs:u4});
			
			\draw [->] ({pic cs:v10l}) to ({pic cs:v02r});
			\draw [->] ({pic cs:v11l}) to ({pic cs:v03r});
			\draw [->] ({pic cs:v12l}) to ({pic cs:v04r});
			
			
			\draw [->] ({pic cs:v20l}) to ({pic cs:v12r});
			\draw [->] ({pic cs:v21l}) to ({pic cs:v13r});
		\end{tikzpicture}
		\caption{Iteration of the restriction of initial data}
		\label{iteration process for eft data}
	\end{table}

	\begin{remark}
		Just as in \cite{reall_effective_2022} Lemma 8 is an if and only if statement, the reverse implication in the above theorem holds too and it follows the steps backwards essentially. For further details see \cite{reall_effective_2022}.
	\end{remark}
	\begin{remark}
		We only established a criteria for the initial conditions of \cref{global existence for higher derivatives} for even values of $n$. The odd case follows similarly, but we restrict the initial data in \cref{exact eft initial data} at the top order to hold for one lower power of $M$.
	\end{remark}

	\begin{lemma}[Long term behaviour of $V^m$]\label{long term behaviour of Vm}
		Fix $n\geq0$ and initial data for \cref{eom} as specified in \cref{eft initial data lemma}. Then the global solutions defined by \cref{global existence for higher derivatives} satisfy the bounds
		\begin{equation*}
			\begin{aligned}
				M^{2m+2}\norm{\partial_t^sV^m}_{H^{N+2n-2m-s}},M^{2m+2}\jpns{t}^{1+3\delta}\norm{\partial_t^sV^m}_{W^{k+2n-2m-s,(1/6-\delta)^{-1}}}\lesssim1/M &&&& \forall s\leq 2n-2m
			\end{aligned}
		\end{equation*}
		for all $m\leq n$. Furthermore, if $n\geq1$ we also have the improvement
		\begin{equation*}
			\begin{aligned}
				M^{2m+3}\norm{\partial_t^sV^m}_{H^{N+2n-2m-2-s}},M^{2m+3}\jpns{t}^{1+3\delta}\norm{\partial_t^sV^m}_{W^{k+2n-2m-2-s,(1/6-\delta)^{-1}}}\lesssim1/M &&&& \forall s\leq 2n-2m-2
			\end{aligned}
		\end{equation*}
		for all $m\leq n-1$.
	\end{lemma}	
	\begin{proof}
		We start with the first estimate and proceed by induction. The case $m=0$ is the statement of \cref{global existence for higher derivatives}. For $m>0$, 
		\begin{equation}\label{Vm equation}
			\begin{aligned}
				(\Box-1)V^m=\frac{{m+1}}{M^{2m}}			
			\end{aligned}
		\end{equation}
		, where $F_{m+1}$ contains derivatives of $U$ up to $2m+1$-st order. Using the energy bound and decay statements of \cref{global existence for higher derivatives} for $U$, we find that the time derivative of the profile of $V^m$ is integrable in time in $H^{N+2n-2m}$. Similarly, the decay statement follows from the analysis of the half-wave components, exactly as in \cref{global existence}. The time derivative terms follow similarly by acting with $\partial_t^s$ on \cref{Vm equation}.
		
		The second part follows from the definition of $V^m=V^{m-1}-F_m[U]/M^{2m}$. We can rearrange this expression for $V^{m-1}$ take norms and observe that both $V^m$ and $F_m[U]/M^{2m}$ are at least order $\mathcal{O}(M^{-2m-2})$.
	\end{proof}
		\begin{remark}\label{dropping v terms}
		As mentioned in the beginning of this section, additional forcing terms on the $V$ equation are not of a problem. In particular, in such a case, the above theorem is not just an energy estimate, instead one has to perform a bootstrap argument, but as both the decay and the $M$ powers are more favourable there's no real barrier to conclude the above lemma in those cases either.
	\end{remark}
	
`	\subsection{EFT solution}
	
	Having established non-trivial initial data that satisfies \cref{eft initial data}, we can look at the error terms in such solutions. In order to describe the end state we need to define what we mean by an EFT solution.
	\begin{definition}[Finite regularity analogue of Definition 2 \cite{reall_effective_2022}]\label{eft sol definition}
		Fix an equation 
		\begin{equation}\label{equations 1}
			\begin{gathered}
				(\Box-1)U_M=\sum_{i=1}^{n}\frac{g_i[U_M]}{M^{i}}
			\end{gathered}
		\end{equation}
		where $g_i$ are local functionals with at most $i$ derivatives of $U_M$ . We call a family of functions $u_M:\R^{3+1}\to\R$, $\sob{\partial_t^su_M}{N-s}\lesssim1$ for $s\leq n$ an EFT solution of order $n$ with regularity $N$ to \cref{equations 1} if there exists $R_M:\R^{3+1}\to\R$ such that $\sob{R_M}{N-n}\lesssim1$ and
		\begin{equation}
			\begin{gathered}
				(\Box-1)u_M=\sum_{i=1}^{n}\frac{g_i[u_M]}{M^{i}}+\frac{R_M}{M^{n+1}}
			\end{gathered}
		\end{equation}
		where the equality holds in $H^{N-n}$. Furthermore, we say that $u_M$ is a scattering EFT solution if $\int_0^\infty\d s\norm{\partial_t^jR(s)_M}_{H^{N-n}}<\infty$ for $j=0,1$ and $\int_0^\infty\d t\norm{\partial_t^a\partial_x^\alpha u}_{L^\infty}<\infty$ for $a+\abs{\alpha}\leq n$.
	\end{definition}
	
	\begin{remark}
		Note, that \cref{eft sol definition} does not specify how one finds the apparently needed $n-1$ initial conditions from the 2 that is given for a usual second order equation. In fact, from \cref{existence of EFT solutions}, we know that those higher derivatives may be determined by finding a solution.
	\end{remark}
	\begin{remark}
		The decay condition imposed on the solution could take alternative forms, such as $\int_0^\infty\d t \norm{g_i}_{H^{N-i}}<\infty$.
	\end{remark}
	
	The reason for the loss of derivatives that is between the spaces for definition of $u$ and where the equation holds is because we cannot make sense of the equation in any higher regularity space. The natural follow-up is to prove that $U$ for \cref{eft eom} is an EFT solution.
	
	\begin{lemma}\label{UV is EFT}
		Fix $n\geq1$, $N,k$ as in \cref{global existence}. Let $\bar{V},\bar{U}$ solve \cref{eom} with initial data as in \cref{eft initial data lemma}, in particular $M^{4}\norm{\bar{V}(0),\partial_t\bar{V}(0)}_{M,N+2n+1,k+2n+1},M\norm{\bar{U}(0),\partial_t\bar{U}(0)}_{1,N+2n+1,k+2n+1}\lesssim1$. Then $U=M\mathcal{U}$ is an EFT solution of regularity $N$ for
		\begin{equation}\label{eft for U}
			(\Box-1)u=-u^3/M^2-u\sum_{i=1}^n\frac{F_i[u]}{M^{2i+2}}    
		\end{equation}
	\end{lemma}
	\begin{proof}
		As shown in \cref{eft eom}, $MU$ indeed solves the EFT equation \cref{eft for U} with error term $M^{2n+3}V^mMU$. The energy and decay estimates for $V^m$ (\cref{long term behaviour of Vm}) and $U$ (\cref{global existence for higher derivatives}) show that
		\begin{equation*}
			\begin{aligned}
				M^{2n+3}\int_0^\infty\d s\norm{V^mMU}_{H^{N}}\lesssim1.
			\end{aligned}
		\end{equation*}
	\end{proof}

	Following \cite{reall_effective_2022}, one can show that EFT solutions do exist generally, without connecting it to UV theory.
	\begin{lemma}\label{existence of EFT solutions}
		Fix $n\geq0$ and $N,k$ as in \cref{global existence}. Given any equation as \cref{equations 1} with $g_i$ at least quadratic, and initial conditions $U_0,U_1$ with $\norm{U_0,U_1}_{1,N+2n,k+2n}\leq1$ there exists a scattering EFT solution of regularity $N$.
	\end{lemma}
	\begin{proof}
		We progress inductively and let $\delta$ be as in \cref{global existence}. The case $n=0$ is solved by linear theory. Let's inductively set $\tilde{U}_j$ ($j\leq n$) to be the maximal in time solutions of 
		\begin{equation}\label{equation in eft existence}
			\begin{gathered}
				(\Box-1)\tilde{U}_j=\sum_{i=1}^{j}\frac{g_i[\tilde{U}_{j-i}]}{M^{i}}\\
				\tilde{U}_j(0)=U_0,\partial_t\tilde{U}_j(0)=U_1.
			\end{gathered}
		\end{equation}
		\begin{claim}
			$\tilde{U}_j$ are global solutions and
			$\norm{\partial_t^s\tilde{U}_j}_{H^{N+2n-j-s}},\jpns{t}^{1+3\delta}\norm{\partial_t^s\tilde{U}_j}_{W^{k+2n-j-s}}\lesssim_j1$ for $s\leq n-j$.
		\end{claim}
		$\tilde{U}_0$ has global solutions and by the induction hypothesis, it also has good global regularity and decay as given by the bootstrap in \cref{bootstrap assumption} with $N,k$ replaced by $N+2n,k+2n$. By induction hypothesis, we know that $\tilde{U}_j$ have good decay properties and the right hand side of \cref{equation in eft existence} is well defined in $H^{N+2n-j}$. Therefore, using linear theory, we get that the solution is global and 
		energy bound and decay are satisfied. To get the larger range of $s$, we simply differentiate the equation with respect to time and obtain similar results.
		
		\begin{claim}
			Let $i\leq j$, then $\norm{\partial_t^s(\tilde{U}_j-\tilde{U}_i)}_{H^{N+2n-j-s}},\jpns{t}^{1+3\delta}\norm{\partial_t^s(\tilde{U}_j-\tilde{U}_i)}_{W^{k+2n-j-s}}\lesssim1/M^{i+1}$
		\end{claim}
		We prove by induction for $s=0$, the $s>0$ case follows similarly after differentiation. For $i=0<j=1$ the claim follows since $U_i-U_j$ has trivial initial data and the (integrable) forcing is of order $1/M$. 
		
		Let's assume the claim hold for $\tilde{i}\leq i,\tilde{j}\leq j,\tilde{i}+\tilde{j}<i+j$. Then, $U_i-U_j$ has trivial initial data and using the induction, we have that the forcing terms $(F_l[\tilde{U}_{i-l}]-F_l[\tilde{U}_{j-l}])/M^l,\, \forall j\leq i$ are of order $1/M^{i+1}$ in the respective norms. Hence we conclude that the claim holds for $i,j$ too.
		
		Using the above two claims, we may insert $\tilde{U}_n$ in place of $\tilde{U}_j$ in the definition of $\tilde{U}_n$ with introducing an error term of order $1/M^{n+1}$, but the equation is only going to be well defined in $H^{N}$ even though $\tilde{U}_N\in H^{N+n}$.
	\end{proof}
		
	In \cite{reall_effective_2022}, the authors showed that EFT solution of a given equation is unique up to an admissible (high order in $M^{-1}$) error term. We now show that an extension of this to scattering also holds
	\begin{lemma}[Uniqueness of scattered states]\label{uniqueness of scattering state}
		For two scattering EFT solutions of order $n$ with regularity $N$ ($N-n>2$) $U,\tilde{U}$ of \cref{equation in eft existence}, let $V,\tilde{V}$ denote the scattered states. Then $\norm{V-\tilde{V}}_{H^{N-n}}\lesssim1/M^{n+1}$ for all times.
	\end{lemma}
	\begin{proof}
		Let $F,\tilde{F}$ be the profiles corresponding to the half wave solution of $U,\tilde{U}$, and set $\psi=F-\tilde{F}$. Then, we have that $\psi$ satisfies an evolution equation with forcing proportional to itself and and an error term $\frac{R-\tilde{R}}{M^{n+1}}$. Taking $H^{N-n}$ norms, we get
		\begin{equation*}
			\begin{aligned}
				\norm{\psi}_{H^{N-n}}\leq \sum_{i=1}^n\frac{\norm{\psi}_{H^{N-n}}a_i(t)}{M^{i}}+\frac{\norm{R-\tilde{R}}_{H^{N-n}}}{M^{n+1}},
			\end{aligned}
		\end{equation*}
		where $a_i(t)$ are $L^\infty$ norms of polynomials in derivatives of $U,\tilde{U}$, thus are integrable. Since $\psi$ has trivial data, using Gronwall's inequality yields the desired result.
	\end{proof}

	\appendix
	\section{Auxiliary tools}\label{auxiliary tools}
	We will encounter many time integrals for which, we need a good estimate:
	\begin{lemma}\label{integral estimate}
		Fix  $\alpha,\beta>0$.
		
		$M$ independent estimate, 
		\begin{equation}
			\begin{gathered}
				\intt \frac{1}{s^\alpha\jpns{t-s}^\beta}\lesssim\begin{cases}
					t^{1-\alpha-\beta} & \alpha,\beta<1\\
					t^{-\alpha} & \beta>\alpha,1\\
					t^{-\beta} & \alpha>\beta,1
				\end{cases} 
			\end{gathered}
		\end{equation}
		
		$M$ dependent
		\begin{equation}
			\begin{gathered}
				\intt \frac{1}{s^\alpha\jpns{(t-s)/M}^\beta}\lesssim M^{\min(1-\alpha,0)}\begin{cases}
					t^{1-\alpha}\jpns{t/M}^{-\beta} & \alpha,\beta<1\\
					\jpns{t/M}^{-\beta} & \alpha>\beta,1\\
					\max(\jpns{t/M}^{-\beta},\jpns{t/M}^{-1}t^{1-\alpha}) & \beta>\alpha,1
				\end{cases} 
			\end{gathered}
		\end{equation}			
	\end{lemma}
	\begin{remark}
		Along with this, we will use many times $\jpns{t}< M\jpns{t/M}$ for $M>1$.
	\end{remark}

	The main dispersive estimate that we are going to use, is similar to the non-relativistic limit, $(\partial_t^2/c^2-\nabla+m^2c^2)u=0$ with the substitution $t/c=\tau$.
	\begin{lemma}[Decay estimate]\label{decay estimate}
		\begin{equation*}
			\begin{aligned}
				\norm{e^{i\jpns{D}_Mt}f}_{L^\infty}\lesssim\jpns{\frac{m}{t}}^{3/2}\norm{f}_{W^{3,1}}				
			\end{aligned}
		\end{equation*}
	\end{lemma}
	\begin{proof}
		For $t<m$ the estimate follows from Sobolev embedding and unitarity.
		
		For	$t>m$, one mimics the treatment of Hormander \cite{hormander_lectures_1997}. It is sufficient to consider
		\begin{equation*}
			\begin{aligned}
				\int \d\rho e^{i(\jpns{\rho}_Mt+\rho x)}\phi(\rho)
			\end{aligned}
		\end{equation*}
		for $\phi\in S_{N}$, that is $\abs{\partial_\rho^\alpha\phi}\leq\rho^{N-\abs{\alpha}}$ for $N<-3/2-1$. The leading order contribution is given by the stationary phase: $\partial_\rho (\jpns{\rho}_mt+\rho x)|_{\rho=\eta}=0\implies \eta/m=-x/(t^2-x^2)$. At the stationary phase, we have $\det(\partial_{\rho_i}\partial_{\rho_j}(\jpns{\rho}_mt+\rho x))|{\rho=\eta}=t^3\frac{1}{\jpns{\eta}_m^3\jpns{\eta/m}^2}$, which gives the estimate, for details, see \cite{hormander_lectures_1997}.
	\end{proof}
	
	\begin{lemma}[Sobolev embedding]\label{sobolev embedding}
		\begin{equation}
			\begin{gathered}
				\norm{f}_{L^p}\lesssim_{p,q}\norm{f}_{W^{k,q}}
			\end{gathered}
		\end{equation}
		for $p,q\in(1,\infty)$, $p>q$, $\frac{d}{p}=\frac{d}{q}-k$, $f\in\mathcal{S}$.
	\end{lemma}
	
	\begin{lemma}[Interpolation estimates]\label{interpolation estimate}
		\begin{equation}
			\begin{gathered}
				\norm{f}_{W^{k,p}}\lesssim_{q_i,r}\norm{f}_{W^{0,q_1}}^r\norm{f}_{W^{\alpha,q_2}}^{1-r}
			\end{gathered}
		\end{equation}
		for $p,q_1,q_2\in(1,\infty)$, $q_1>p\geq2\geq q_2$, $\frac{1}{p}=\frac{r}{q_1}+\frac{1-r}{q_2}$,$\alpha(1-r)=k$, $f\in\mathcal{S}$.
	\end{lemma}
	
	\begin{lemma}[Product estimate]\label{product estimate}
		\begin{equation}
			\begin{gathered}
				\norm{fg}_{W^{k,p}}\lesssim_{k,p}\norm{f}_{W^{k,p}}\norm{g}_{L^\infty}+\norm{f}_{L^\infty}\norm{g}_{W^{k,p}}
			\end{gathered}
		\end{equation}
		for all $f,g\in\mathcal{S}$ $k,p-1\geq 0$.
	\end{lemma}
	
	\begin{lemma}[Coifman-Meyer multipliers]\label{Coifman-Meyer}
		Let $m$ be a \textit{Coifman-Meyer} multiplier, so it satisfies
		\begin{equation}
			\begin{gathered}
				\abs{\nabla_\rho^j\nabla_\nu^k m(\rho,\nu)}\lesssim_{j,k}(\abs{\nu}+\abs{\rho})^{-j-k}
			\end{gathered}
		\end{equation}
		for $j+k$ smaller than some dimension dependent constant, then
		\begin{equation}
			\begin{gathered}
				\norm{T_m(f,g)}_{L^r}\lesssim_{p,q,m}\norm{f}_{L^p}\norm{g}_{L^q}
			\end{gathered}
		\end{equation}
		for $f,g\in\mathcal{S}$, $\frac{1}{r}=\frac{1}{p}+\frac{1}{q}$.
	\end{lemma}
	
	\pagebreak
	\printbibliography
\end{document}